\documentclass[11pt]{amsart}

\usepackage{mathrsfs} 
\usepackage{mathtools}
\usepackage{a4wide} 
\usepackage{graphicx} 
\usepackage{amssymb} 
\usepackage{epstopdf}
\usepackage{enumerate} 
\usepackage{titletoc} 
\usepackage[colorlinks=true, urlcolor=blue, linkcolor=blue, citecolor=black]{hyperref} %
\usepackage[nameinlink]{cleveref} 
\usepackage{esint} 
\usepackage{verbatim}
\usepackage{enumitem}

\numberwithin{equation}{section}

\theoremstyle{plain}
\newtheorem{theorem}{Theorem}[section]
\newtheorem{lemma}[theorem]{Lemma}
\newtheorem{corollary}[theorem]{Corollary}

\theoremstyle{definition}
\newtheorem{definition}[theorem]{Definition}
\newtheorem{remark}[theorem]{Remark}

\newcommand{\ssubset}{\subset\joinrel\subset}

\makeatletter
\@namedef{subjclassname@2020}{\textup{2020} Mathematics Subject Classification}
\makeatother

\title[Fully nonlinear degenerate/singular parabolic equations]{Boundary Regularity for viscosity solutions of Fully nonlinear degenerate/singular parabolic equations}

\author{Ki-Ahm Lee}
\address{Department of Mathematical Sciences and Research Institute of Mathematics,
	Seoul National University, Seoul 08826, Republic of Korea.}
\email{kiahm@snu.ac.kr}

\author{Hyungsung Yun}
\address{Department of Mathematical Sciences, Seoul National University, Seoul 08826, Republic of Korea}
\email{euler@snu.ac.kr}


\thanks{Ki-Ahm Lee is supported by the National Research Foundation of Korea (NRF) grant: NRF2021R1A4A1027378. Ki-Ahm Lee also holds a joint appointment with the Research Institute of Mathematics of Seoul National University.}

\begin{document}

\begin{abstract}
In this paper, we establish the boundary regularity results for viscosity solutions of fully nonlinear degenerate/singular parabolic equations of the form
$$u_t - x_n^{\gamma} F(D^2 u,x,t) = f,$$
where $\gamma<1$. These equations are motivated by the porous media type equations. We show the boundary $C^{1,\alpha}$-regularity of functions in their solutions class and the boundary $C^{2,\alpha}$-regularity of solutions. 
As an application, we derive the global regularity results and the solvability of the Cauchy-Dirichlet problems. 
\end{abstract}

\maketitle

%
%

\section{Introduction}
In this paper, we study regularity results for viscosity solutions of the following fully nonlinear degenerate/singular parabolic equations
\begin{equation} \label{eq:model}
	u_t - x_n^{\gamma} F(D^2 u,x,t) = f,
\end{equation}
where $\gamma < 1$ and $f$ is a bounded continuous function. Here the fully nonlinear operator F is uniformly parabolic with certain structural conditions (the hypotheses on F will be precisely stated in \Cref{sec:pre}). 

The fully nonlinear parabolic equations \eqref{eq:model} are motivated by the porous media type equations 
\begin{equation} \label{eq:FPME} 
	\left\{\begin{aligned}
		u_t - F(D^2 u^m)&=0 && \text{in } \Omega \times (0,\infty) \\
		u &= u_0 && \text{in } \Omega \times \{t=0\}  \\
		u &= 0 && \text{on } \partial \Omega \times [0,\infty),
	\end{aligned}\right.
\end{equation}
where $\Omega\subset\mathbb{R}^n$ is a smooth bounded domain and $u_0>0$ in $\Omega$. By setting $v=u^m$, the Cauchy-Dirichlet problems \eqref{eq:FPME} can be transformed into the following form
\begin{equation} \label{eq:FPME_tr}
	\left\{\begin{aligned}
		v_t - v^{\gamma} \widetilde{F}(D^2 v)&=0 && \text{in } \Omega \times (0,\infty) \\
		v &= u_0^m && \text{in } \Omega \times \{t=0\}  \\
		v &= 0 && \text{on } \partial \Omega \times [0,\infty),
	\end{aligned}\right.
\end{equation}
where $\gamma = 1 -1/m \in (-\infty, 1)$ and $\widetilde{F}(M)=m F(M)$; we refer to \cite[Section 5]{KL13}. In their work \cite{DH98,KL09}, it is demonstrated that \eqref{eq:model} can be derived by applying the hodograph transform to \eqref{eq:FPME_tr}. Moreover, they investigate the regularity results for solutions and free boundary of the porous medium equation by analyzing the regularity for solutions of \eqref{eq:model}. Besides the porous medium equation, degenerate/singular equations on flat boundaries appear in various nonlinear problems with applications; for example, Monge--Ampère equations, Gauss curvature flow, extension problems related to non-local equations, and mathematical finance, etc; see e.g. \cite{CS07,DH99,DS09,KLY24,KLY23,LS17}.

Our main results are to prove the boundary $C^{1,\alpha}$-regularity and the boundary $C^{2,\alpha}$-regularity for viscosity solutions. Unlike the interior $C^{1,\alpha}$-regularity, the boundary $C^{1,\alpha}$-regularity does not require the regularity of the operator in the variables $x$ and $t$, so it can be proved for functions in the solutions class $S^*(\lambda,\Lambda, f)$ suggested in \cite{CC95}. In detail, the solutions class $S^*(\lambda,\Lambda, f)$ is the set of functions that satisfy the following inequalities
\begin{equation}\label{eq:main}
	\left\{\begin{aligned}
		u_t - x_n^{\gamma}\mathcal{M}^{-}_{\lambda, \Lambda}(D^2u) & \ge- \|f\|_{L^{\infty}(Q_1^+)} \\
		u_t - x_n^{\gamma}\mathcal{M}^{+}_{\lambda, \Lambda}(D^2u)  & \leq \|f\|_{L^{\infty}(Q_1^+)} 
	\end{aligned}\right.
	\quad \text{in $Q_1^+$},
\end{equation}
where $\gamma < 1$, $f \in C(Q_1^+) \cap L^{\infty}(Q_1^+)$, and $\mathcal{M}_{\lambda, \Lambda}^{\pm}$ are the Pucci's extremal operators. (The operators $\mathcal{M}_{\lambda, \Lambda}^{\pm}$ will be precisely defined in \Cref{def_pucci}.)

A standard method to establish the boundary $C^{1,\alpha}$-regularity for functions in the solutions class of uniformly parabolic equations has been presented; see \cite{LZ22,Wan92b} for details. This method utilizes Lipschitz estimate and Hopf principle to sandwich the solution $u$ between the two linear functions and applies it repeatedly through scaling to find a linear function that approximates $u$. Moreover, the boundary $C^{1,\alpha}$-regularity has been investigated for uniformly elliptic equations by means of this method, as discussed in \cite{LZ20,SS14}. On the other hand, the interior regularity for functions in the solutions class is at most $C^{\alpha}$ in general; we refer to \cite{NV07, NV08, NV10}. In fact, since the approximation lemma is used in the proof for the interior $C^{1,\alpha}$-regularity, the regularity of $F$ in the variables $x$ and $t$ is required, so solutions for the given $F$ rather than functions in solutions class is considered; we refer to \cite{CC95,Wan92b}. Nonetheless, there is an interior $C^{1,\alpha}$-regularity result for functions in the solutions class under stringent conditions. In \cite{LY23}, they proved that for a given $\alpha \in (0,1)$, there exists $\delta>0$ depending only on $n$ and $\alpha$ such that if $\Lambda \leq (1+\delta)\lambda$, then functions in the solutions class becomes $C^{1,\alpha}$ at the interior point.

Our first main theorem is concerned with the boundary $C^{1, \alpha}$-regularity for viscosity solutions $u$ of \eqref{eq:main}. The constant $\overline{\alpha} \in (0,1)$ used in describing theorems throughout this paper is a universal constant derived from \Cref{lem:c1a_f0_g0}.
\begin{theorem}[Boundary $C^{1, \alpha}$-estimate] \label{thm:main}
	Let $\alpha \in(0, \overline{\alpha})$ with $\alpha \leq 1-\gamma$ and $(x_0,t_0) \in \partial_p Q_1^+ \cap \{x_n=0\}$. Suppose that $f\in C(Q_1^+) \cap L^{\infty}(Q_1^+)$, $g\in C^{1,\alpha}(x_0,t_0)$, and $u \in C(\overline{Q_1^+})$ satisfy \eqref{eq:main} with $u = g$ on $\partial_p Q_1^+ \cap \{x_n=0\}$. Then $u \in C^{1, \alpha}(x_0,t_0)$, i.e., there exists a linear function $L$ such that
	\begin{align*}
		|u(x,t)-L(x)|\leq C(\|u\|_{L^{\infty}(Q_1^+)}+\|f\|_{L^{\infty}(Q_1^+)}+ \|g\|_{C^{1,\alpha}(x_0,t_0)})(|x-x_0|+|t-t_0|^{\frac{1}{2-\gamma}})^{1+\alpha} 
	\end{align*}
	for all $(x,t) \in Q_1^+ \cap Q_1^+(x_0,t_0)$ and 
	\begin{equation*}
		|Du(x_0,t_0)| \leq C(\|u\|_{L^{\infty}(Q_1^+)}+\|f\|_{L^{\infty}(Q_1^+)}+ \|g\|_{C^{1,\alpha}(x_0,t_0)}),
	\end{equation*}
	where $C>0$ is a constant depending only on $n$, $\lambda$, $\Lambda$, $\gamma$, and $\alpha$.
\end{theorem}
The following remark shows that solutions of \eqref{eq:main} cannot be more regular than $C^{1,1-\gamma}$ when $0<\gamma <1$.
\begin{remark} \label{rmk1}
Let $0<\gamma<1$. For $\varepsilon \in (0,\gamma]$, consider the equations $u_t -x_n^{\gamma} \Delta u = x_n^{\gamma-\varepsilon}$ which has a solution of the form
\begin{equation*}
	u(x,t)=1-\frac{x_n^{2-\varepsilon}}{(2-\varepsilon)(1-\varepsilon)} .
\end{equation*}
Since $|Du|=x_n^{1-\varepsilon}/|1-\varepsilon|$, the solutions $u$ is $C^{1,1-\varepsilon}$ on the flat boundary $\{x_n=0\}$. This example shows that even for a H\"older continuous $f$ and the Laplace operator, we cannot expect in general the solution of \eqref{eq:model} to be more regular than $C^{1,1-\gamma}$ on $\{x_n=0\}$ when $0<\gamma <1$.
\end{remark}
Combining \Cref{thm:main} and the known $C^{1,\alpha}$-regularity results (see \cite[Theorem 1.8]{LZ22}, \cite[Theorem 1.3]{Wan92b}) gives the following global $C^{1,\alpha}$-regularity. For this purpose, $C^{1,\alpha}$-estimate near the boundary is required, and this result can be found in \cite[Lemma 3.10]{LLY23}. 
\begin{corollary} [Global $C^{1, \alpha}$-regularity]\label{cor:global}
Let $\alpha_1 \in (0,\overline{\alpha})$ with $\alpha_1 \leq 1-\gamma$, and let $\Omega \in \mathbb{R}^{n+1}$ be a bounded domain with $(0,0)\in \Omega$ and $\partial_p \Omega^+ \in C^{1,\alpha_1}$. Assume that  $f\in C(\Omega^+) \cap L^{\infty}(\Omega^+)$, $g\in C^{1,\alpha_1}(\partial_p \Omega^+)$, $F$ satisfies \ref{F1} and \ref{F4}, and $u \in C(\overline{\Omega^+})$ is a viscosity solution of  
\begin{equation} \label{eq:poisson}
	\left\{\begin{aligned}
		u_t - x_n^{\gamma} F(D^2 u,x,t) &= f && \text{in } \Omega^+ \\
		u & = g && \text{on } \partial_p \Omega^+ .
	\end{aligned}\right.
\end{equation}
Then there exists a constant $\alpha \in (0,\alpha_1]$ depending only on $n$, $\lambda$, $\Lambda$, and $\gamma$ such that $u \in C^{1,\alpha}(\overline{\Omega^+})$ with the uniform estimate 
\begin{equation*}
	\|u\|_{C^{1,\alpha} (\overline{\Omega^+})} \leq C  ( \|u\|_{L^{\infty}(\Omega^+)} + \|f\|_{L^{\infty} (\Omega^+)} +\|g\|_{C^{1,\alpha} (\partial_p \Omega^+)} ) ,
\end{equation*}
where $C>0$ is a constant depending only on $n$, $\lambda$, $\Lambda$, $\gamma$, $\alpha$, $C_{\theta}$, and $\|\partial_p \Omega^+\|_{C^{1,\alpha_1}}$.
\end{corollary}

Our second main theorems are concerned with the boundary $C^{2, \alpha}$-regularity of viscosity solutions $u$ of \eqref{eq:model}. Unlike the boundary $C^{1,\alpha}$-regularity, the $C^{2,\alpha}$-regularity of solutions depends on the regularity of the operator $F$ and forcing term $f$. This study  investigated how the conditions of $f$ relate to the $C^{2,\alpha}$-regularity for solutions of \eqref{eq:model} depending on the value of $\gamma$.
{
In order to obtain the heuristic idea, let's first consider the degenerate elliptic equation $x_n^{\gamma} F(D^2 u,x) =f$ for $\gamma \in(0,1)$. This equation can be understood as a uniformly elliptic equation $F(D^2 u,x)=x_n^{-\gamma} f$, so if $F(\cdot,x)$ is $C^{\alpha}$ at $x=0$ and $f$ is $C^{\alpha+\gamma}$ at $x=0$, we have $u$ is $C^{2,\alpha}$ at $x=0$. In fact, it can be seen that the conditions of $f$ obtained heuristically for elliptic equations are expandable for \eqref{eq:model} through this study.

Recall that the optimal regularity for solutions of \eqref{eq:poisson} was at most $C^{1,1-\gamma}$ when $0<\gamma<1$; see \Cref{rmk1}. When $f \equiv 0$ in \eqref{eq:poisson}, we can expect higher regularity for \eqref{eq:poisson} under the assumptions of certain structure hypotheses (the hypotheses will be precisely stated in \Cref{sec:shf}). Actually, since $x_n^{\gamma} F(O_n,x,t)$ plays the role of $f$, it is the same as considering $f$ where $x_n^{-\gamma}f$ is $C^{\alpha}$ at $(0,0)$.

\begin{remark} \label{rmk2}
Consider the equation $u_t -x_n^{\gamma} \Delta u =0 $ which has a solution of the form
\begin{equation*}
	u(x,t)= l(x) + tx_n + \frac{x_n^{3-\gamma}}{(3-\gamma)(2-\gamma)} ,
\end{equation*}
where $l(x)$ is a linear function. Since $\|D^2u\|= x_n^{1-\gamma} $, the solutions $u$ is $C^{2,1-\gamma}$ on the flat boundary $\{x_n=0\}$. This example shows that we cannot expect in general solutions of 
\begin{equation}\label{eq:homo_int}
	u_t - x_n^{\gamma} F(D^2 u,x,t) = 0
\end{equation}
 to be more regular than $C^{2,1-\gamma}$ on $\{x_n=0\}$ when $0<\gamma <1$. Also, $u_t(x',0,t)=0$ is an essential condition of the existence of solutions $u\in C^{2}(\overline{Q_1^+})$ of \eqref{eq:homo_int}. Indeed, if $u\in C^{1,1}(\overline{Q_1^+})$ then $|F(D^2u,x,t)|\leq C$ on $\overline{Q_1^+}$. It follows that 
\begin{equation*}
 	|u_t(x,t)| =x_n^{\gamma}|F(D^2u,x,t)|\leq Cx_n^{\gamma} \quad \text{for all } (x,t) \in \overline{Q_1^+},
\end{equation*}
when $0<\gamma<1$.
\end{remark}

We now introduce the second main theorem for the degenerate equations.
\begin{theorem}[Boundary $C^{2, \alpha}$-estimate for $0 < \gamma <1$] \label{thm:main2}
	Let $\alpha \in(0, \overline{\alpha})$ with $\alpha \leq 1-\gamma$, $(x_0,t_0) \in \partial_p Q_1^+ \cap \{x_n=0\}$, and let  $F$ satisfies \ref{F1}, \ref{F2}, and \ref{F3}. Suppose that $g\in C^{2,\alpha}(x_0,t_0)$ with $g_t(x_0,t_0)=0$, and $u \in C(\overline{Q_1^+})$ is a viscosity solution of 
\begin{equation*} 
	\left\{\begin{aligned}
		u_t - x_n^{\gamma} F(D^2 u,x,t) &= 0 && \text{in } Q_1^+ \\
		u & = g && \text{on } \partial_p Q_1^+ \cap \{x_n=0\}. 
	\end{aligned}\right.
\end{equation*}
Then $u \in C^{2, \alpha}(x_0,t_0)$, i.e., there exists a polynomial $P$ with $\deg P \leq 2$ such that
	\begin{align*}
		|u(x,t)-P(x)|\leq CN(|x-x_0|+|t-t_0|^{\frac{1}{2-\gamma}} )^{2+\alpha} 
	\end{align*}
	for all $(x,t) \in Q_1^+ \cap Q_1^+(x_0,t_0)$, $F(D^2 P,x_0,t_0) =0$, $P(x',0) \equiv P_g(x',0)$, and 
	\begin{equation*}
		u_t(x_0,t_0)=0, \quad |Du(x_0,t_0)| + \|D^2u(x_0,t_0)\| \leq CN,
	\end{equation*}
	where the constant $N$ is given by
	\begin{equation*}
		N\coloneqq (\|u\|_{L^{\infty}(Q_1^+)} + \|g\|_{C^{2,\alpha}(x_0,t_0)}+ [\beta^2]_{C^{\alpha}(x_0,t_0)}+ |F(O_n,x_0,t_0)|)
	\end{equation*}
	and $C>0$ is a constant depending only on $n$, $\lambda$, $\Lambda$, $\gamma$, $\alpha$, and $[\beta^1]_{C^{\alpha}(x_0,t_0)}$.
\end{theorem}


To prove the $C^{2,\alpha}$-regularity of solutions to uniformly parabolic equations, it suffices to have $f \in C^{\alpha}(x_0,t_0)$, which is a less stringent condition than the condition in \Cref{thm:main2} for degenerate equations. From the following theorem, we know that the $C^{2,\alpha}$-regularity of solutions to singular equations can be achieved with lower regularity than $f \in C^{\alpha}(x_0,t_0)$, even if it is discontinuous. The following result is our third main theorem for the singular equations.
\begin{theorem}[Boundary $C^{2, \alpha}$-estimate for $\gamma \leq 0$] \label{thm:main2_sg}
	Let $\alpha \in(0, \overline{\alpha})$, $(x_0,t_0) \in \partial_p Q_1^+ \cap \{x_n=0\}$, and let  $F$ satisfies \ref{F1}, \ref{F2}, and \ref{F3}. Suppose that  
	\begin{equation} \label{cond_sg}
		|f(x,t)-f(x_0,t_0)| \leq K (|x-x_0|+|t-t_0|^{\frac{1}{2-\gamma}})^{\alpha+\gamma} \quad \text{for all } (x,t) \in \overline{Q_1^+} ,
	\end{equation}
	$g\in C^{2,\alpha}(x_0,t_0)$, and $u \in C(\overline{Q_1^+})$ is a viscosity solution of 
\begin{equation} \label{eq:c2a_sg}
	\left\{\begin{aligned}
		u_t - x_n^{\gamma} F(D^2 u,x,t) &= f && \text{in } Q_1^+ \\
		u & = g && \text{on } \partial_p Q_1^+ \cap \{x_n=0\}. 
	\end{aligned}\right.
\end{equation}
Then $u \in C^{2, \alpha^*}(x_0,t_0)$, i.e., there exists a polynomial $P$ with $\deg P \leq 2$ such that
	\begin{align*}
		|u(x,t)-P(x,t)|\leq CN(|x-x_0|+|t-t_0|^{\frac{1}{2-\gamma}} )^{2+\alpha^*} 
	\end{align*}
	for all $(x,t) \in Q_1^+ \cap Q_1^+(x_0,t_0)$, $F(D^2 P,x_0,t_0) =0$, $\partial_t P =\partial_t P_g$, $P(x',0,t) \equiv P_g(x',0,t)$, and 
	\begin{equation*}
		|u_t(x_0,t_0)| + |Du(x_0,t_0)| + \|D^2u(x_0,t_0)\| \leq CN,
	\end{equation*}
	where the constant $N$ is given by
	\begin{equation*}
		N\coloneqq (\|u\|_{L^{\infty}(Q_1^+)} + \|g\|_{C^{2,\alpha}(x_0,t_0)}+ [\beta^2]_{C^{\alpha}(x_0,t_0)} + |F(O_n,x_0,t_0)| + K),
	\end{equation*}
	$a^* \coloneqq \min\{\alpha, -\gamma\}$, and $C>0$ is a constant depending only on $n$, $\lambda$, $\Lambda$, $\gamma$, $\alpha$, and $[\beta^1]_{C^{\alpha}(x_0,t_0)}$.
\end{theorem}

Finally, such boundary $C^{2,\alpha}$-regularity can be applied to describe the solvability of the Cauchy-Dirichlet problems \eqref{eq:poisson} and the short time existence of solutions for \eqref{eq:FPME_tr}; see \cite[Theorem]{DH98} and \cite[Theorem 1.1]{KL09} for the short time existence.
\begin{theorem}[The solvability of the Cauchy-Dirichlet problem for $0 < \gamma < 1$] \label{thm:solvability}
	 Let $\alpha_1 \in (0,\overline{\alpha})$ with $\alpha_1 \leq 1-\gamma$ and let $\Omega_T^+ $ be a bounded domain with $(0,0)\in \Omega$ and $\partial \Omega^+ \in C^{2,\alpha_1}$. Suppose that $F$ satisfies \ref{F1}, \ref{F2}, and \ref{F3}. If $g\in C^{2,\alpha_1}(\overline{\Omega_T^+})$ satisfy the compatibility conditions
	\begin{equation*}
		\left\{\begin{aligned}
			g_t -x_n^{\gamma}F(D^2 g,x,t) &=0 && \text{on } \partial_c \Omega_T^+ \\
			g_t &=0 &&  \text{on } \partial_p \Omega_T^+ \cap \{x_n=0\},
		\end{aligned}\right.
	\end{equation*}
	then the Cauchy-Dirichlet problem
	\begin{equation*} 
	\left\{\begin{aligned}
		u_t - x_n^{\gamma} F(D^2 u,x,t) &= 0 && \text{in } \Omega_T^+ \\
		u & = g && \text{on } \partial_p \Omega_T^+   
	\end{aligned}\right.
\end{equation*}
	is uniquely solvable in $H^{2,\alpha}(\overline{\Omega_T^+ })$ for some $\alpha \in (0,\alpha_1]$.
\end{theorem}

Similar to the one described above, the solvability of the Cauchy-Dirichlet problem for singular equations can be derived.  It is also worth emphasizing that while \eqref{cond_sg} provides a sufficient condition for boundary $C^{2,\alpha}$-regularity, it does not guarantee interior $C^{2,\alpha}$-regularity.
\begin{theorem}[The solvability of the Cauchy-Dirichlet problem for $\gamma \leq 0$] \label{thm:solvability_s}
	 Let $\alpha_1 \in (0,\overline{\alpha})$ and let $\Omega_T^+ $ be a bounded domain with $(0,0)\in \Omega$ and $\partial \Omega^+ \in C^{2,\alpha_1}$. Suppose that $F$ satisfies \ref{F1}, \ref{F2}, and \ref{F3'}. If $f \in H^{\alpha_1}(\overline{\Omega_T^+})$ and $g\in C^{2,\alpha_1}(\overline{\Omega_T^+})$ satisfy the compatibility conditions
	\begin{equation*}
			g_t -x_n^{\gamma}F(D^2 g,x,t) =f  \quad  \text{on } \partial_c \Omega_T^+,
	\end{equation*}
	then the Cauchy-Dirichlet problem
	\begin{equation*} 
	\left\{\begin{aligned}
		u_t - x_n^{\gamma} F(D^2 u,x,t) &= f && \text{in } \Omega_T^+ \\
		u & = g && \text{on } \partial_p \Omega_T^+   
	\end{aligned}\right.
\end{equation*}
	is uniquely solvable in $C^{2,\alpha}(\overline{\Omega_T^+ })$ for some $\alpha \in (0,\alpha^*]$.
\end{theorem}

When we investigate the boundary regularity for solutions of \eqref{eq:model}, the essential challenge arises from the fact that the fully nonlinear operator $F$ strongly depends on the term $x_n^{\gamma}$. It causes the ellipticity constants to be zero or infinity near the boundary and makes it difficult to construct a barrier function using the fundamental solution of the heat equation. In the case of $F(D^2u,x,t)=\Delta u$, estimates of $Du$, $u_{ii} \, (1\leq i<n)$, and $u_t$ can be obtained using Bernstein's technique; we refer to \cite{KL09}. They obtain an estimate of $x_n^{\gamma}u_{nn}$ using these estimates and the equation that $u$ satisfies. For smooth and concave operator $F$, which is a natural extension of $\Delta$, the estimate of $u_{in}\,(1\leq i <n)$ is required. However, since the $u_n$ is not a function in the same solutions class, it is difficult to obtain the estimate of $u_{in} \, (1\leq i <n)$ using Bernstein's technique. In contrast, our approach can obtain estimates of these derivatives even under the general $F$. Also, our study accurately captures the variations of assumptions regarding the existence of a classical solution based on the value range of $\gamma$.

The paper is organized as follows. In \Cref{sec:pre}, we summarize several notations, definitions, and known results that will be used throughout the paper. \Cref{sec:c1a} is devoted to the proof of \Cref{thm:main} for the solutions class. \Cref{sec:dirichlet} is devoted to the proof of \Cref{thm:main2} and \Cref{thm:main2_sg} for \eqref{eq:model}. Finally, we provide the applications of \Cref{thm:main2} and \Cref{thm:main2_sg}: The solvability of the Cauchy-Dirichlet problems.
%
%
\section{Preliminaries} \label{sec:pre}
In this section, we summarize some basic notations and gather definitions and known regularity results used throughout the paper. 

Since the standard scaling used in the uniformly parabolic equations cannot be applied to \eqref{eq:model}, it is necessary to define a cylinder and H\"older space accordingly.

For a point $(x_0,t_0) \in \mathbb{R}^{n+1}$ and $r>0$, we denote the intrinsic cylinder and upper cylinder as
\begin{align*}
	Q_r(x_0,t_0) &\coloneqq \{x \in  \mathbb{R}^n :|x-x_0|< r  \}  \times (t_0 -r^{2-\gamma} , t_0],\\
	Q_r^+(x_0,t_0) &\coloneqq Q_r(x_0,t_0) \cap \{ x_n>0\}. 
\end{align*}
For convenience, we denote $Q_r = Q_r(0,0)$ and $Q_r^+ = Q_r^+(0,0)$. 

Now let's define a more general domain. For a bounded domain $\Omega \subset \mathbb{R}^{n+1}$, we denote the upper domain and parabolic boundary as
\begin{align*}
	\Omega^+  &\coloneqq \{(x,t) \in \Omega:x_n>0 \} ,\\
	\partial_p \Omega &\coloneqq \{(x,t) \in \partial \Omega : Q_r(x,t) \cap \Omega^c \text{ is nonempty for all } r >0 \}.  
\end{align*}
In the special case $\Omega_T^+ \coloneqq \Omega^+ \times (0,T]$ for a bounded domain $\Omega \in \mathbb{R}^n$, we define the bottom, corner, and side as
\begin{align*}
	\partial_b \Omega^+_T \coloneqq \Omega^+  \times \{ t= 0\}, \quad
	\partial_c \Omega^+_T \coloneqq \partial \Omega^+ \times \{ t= 0 \}, \quad \text{and} \quad
	\partial_s \Omega^+_T \coloneqq  \partial \Omega^+ \times (0,T).
\end{align*}

\subsection{Hölder Spaces}
Defines the H\"older continuity of functions and domains used in the description of the theorem. Note that unlike commonly used H\"older continuity, it is H\"older continuity suitable for scaling in \eqref{eq:model}.
\begin{definition}[Hölder spaces]
	Let $\Omega \subset \mathbb{R}^{n+1}$ be an open set and $\alpha \in (0,1)$ with $\alpha \leq 1-\gamma$.
	\begin{enumerate} [label=(\roman*)]
		\item $u \in C^{\alpha}(\overline{\Omega})$ means that there exists $C>0$ such that 
		$$|u(x,t)-u(y,s)| \leq C(|x-y|+|t-s|^{\frac{1}{2-\gamma}})^{\alpha} \quad \text{for all } (x,t), (y,s) \in \Omega.$$
		In other words, $u$ is $\frac{\alpha}{2-\gamma}$-H\"older continuous in $t$ and  $\alpha$-H\"older continuous in $x$.
		\item $u \in C^{1,\alpha}(\overline{\Omega})$ means that $u$ is $\frac{1+\alpha}{2-\gamma}$-H\"older continuous in $t$ and $Du$ is $\alpha$-H\"older continuous in $x$.
		\item $u \in C^{2,\alpha}(\overline{\Omega})$ means that $u_t$ is $\frac{\alpha+\gamma}{2-\gamma}$-H\"older continuous in $t$ and $D^2u$ is $\alpha$-H\"older continuous in $x$.
	\item We denote $H^{\alpha}(\overline{\Omega}) \coloneqq C^{\alpha}(\overline{\Omega})$ and $H^{2,\alpha}(\overline{\Omega}) \coloneqq C^{2,\alpha}(\overline{\Omega})$ when $\gamma=0$.
 	\end{enumerate}
\end{definition}

\begin{definition}
	Let $A \in \mathbb{R}^{n+1}$ be a bounded set and $f : A \to \mathbb{R}$ be a function. We say that $f$ is $C^{k,\alpha}$ at $(x_0,t_0) \in A$ (denoted by $f \in C^{k,\alpha}(x_0,t_0)$), if there exist constant $C>0$, $r>0$, and a polynomial $P_f(x,t)$ with $\deg P_f \leq k$ such that 
	\begin{equation}\label{cka_f}
		|f(x,t)-P_f(x,t)| \leq C(|x-x_0|+|t-t_0|^{\frac{1}{2-\gamma}})^{k+\alpha} \quad \text{for all } (x,t) \in A \cap Q_r^+(x_0,t_0).
	\end{equation}
	We define 
	\begin{align*}
		[f]_{C^{k,\alpha}(x_0,t_0)} &\coloneqq \inf \{C>0 \mid \eqref{cka_f} \text{ holds with } P_f(x,t) \text{ and } C \}, \\
		\|f\|_{C^{k,\alpha}(x_0,t_0)} &\coloneqq [f]_{C^{k,\alpha} (x_0,t_0)} + \sum_{i=0}^k \|D^i P_f(x_0,t_0)\|.
	\end{align*}
\end{definition}

\begin{definition}
	Let $\Omega \subset \mathbb{R}^{n+1}$ be a bounded domain. We say that $\partial_p \Omega$ is $C^{k,\alpha}$ at $(x_0,t_0) \in \partial_p \Omega$ (denoted by $\partial_p \Omega \in C^{k,\alpha}(x_0,t_0)$), if there exist constant $C>0$, $r>0$, a new coordinate system $\{x_1,\cdots,x_n,t\}$, and a polynomial $P(x',t)$ with $\deg P \leq k$, $P(0',0)=0$, and $DP(0',0)=0$ such that $(x_0,t_0)=(0,0)$ in this coordinate system,
	\begin{equation}\label{cka_dom}
		\begin{aligned}
			&Q_\rho \cap \{ (x',x_n,t) \mid x_n > P(x',t) + C(|x'|+|t|^{\frac{1}{2-\gamma}})^{k+\alpha} \} \subset Q_r \cap \Omega, \\
			&Q_\rho \cap \{ (x',x_n,t) \mid x_n < P(x',t) - C(|x'|+|t|^{\frac{1}{2-\gamma}})^{k+\alpha} \} \subset Q_r \cap \Omega^c.
		\end{aligned}
	\end{equation}
	We define 
	\begin{align*}
		[\partial_p \Omega]_{C^{k,\alpha}(x_0,t_0)} &\coloneqq \inf \{C>0 \mid \eqref{cka_dom} \text{ holds with } P(x',t) \text{ and } C \}, \\
		\|\partial_p \Omega\|_{C^{k,\alpha}(x_0,t_0)} &\coloneqq [\partial_p \Omega]_{C^{k,\alpha} (x_0,t_0)}  + \sum_{i=2}^k \|D^i P(0',0)\|.
	\end{align*}
Moreover, if $\partial_p \Omega \in C^{k,\alpha}(x,t)$ for all $(x,t) \in \partial_p \Omega$, we define 
\begin{equation*}
	\|\partial_p \Omega\|_{C^{k,\alpha}}\coloneqq \sup_{(x,t)\in\partial_p \Omega} \|\partial_p \Omega\|_{C^{k,\alpha}(x,t)}.
\end{equation*}
\end{definition}

\subsection{Viscosity Solutions}
We now introduce the concept of solutions of fully nonlinear parabolic equations:
\begin{equation}\label{eq:op_F}
	u_t - F(D^2u,Du,x,t) = 0,
\end{equation}
where we assume that the fully nonlinear operator $F:\mathcal{S}^n \times \mathbb{R}^n \times \Omega \to \mathbb{R}$ is continuous and satisfies the following condition:
\begin{equation*}
	M \leq N \quad \Longrightarrow \quad F(M,p,x,t) \leq  F(N,p,x,t). 
\end{equation*}

\begin{definition} [Test functions]
	Let $u$ be a continuous function in $\Omega$. The function $\varphi : \Omega \to \mathbb{R}$ is called \textit{test function} if it is $C^1$ with respect to $t$ and $C^2$ with respect to $x$.
	\begin{enumerate} [label=(\roman*)]
		\item We say that the test function $\varphi$ touches $u$ from above at $(x,t)$ if there exists an open neighborhood $U$ of $(x,t)$ such that 
		$$u \leq \varphi  \quad \mbox{in } U \qquad  \mbox{and} \qquad u(x,t) = \varphi(x,t). $$
		\item We say that the test function $\varphi$ touches $u$ from below at $(x,t)$ if there exists an open neighborhood $U$ of $(x,t)$ such that 
		$$u \ge \varphi  \quad \mbox{in } U \qquad  \mbox{and} \qquad u(x,t) = \varphi(x,t). $$
	\end{enumerate}
\end{definition}

\begin{definition}[Viscosity solutions] \label{vis1}
	Let $F:\mathcal{S}^n \times \mathbb{R}^n \times \Omega \to \mathbb{R}$ be a fully nonlinear parabolic operator.
	\begin{enumerate} [label=(\roman*)]
		\item Let $u$ be a upper semicontinuous function in $\Omega$. $u$ is called a \textit{viscosity subsoution} of \eqref{eq:op_F} when the following condition holds: if for any $(x,t) \in \Omega $ and any test function $\varphi$ touching $u$ from above at $(x,t)$, then
		\begin{equation*}
			\varphi_t (x,t) - F ( D^2 \varphi (x,t),D \varphi (x,t),x,t ) \leq f(x,t).
		\end{equation*}
		\item Let $u$ be a lower semicontinuous function in $\Omega$. $u$ is called a \textit{viscosity supersoution} of \eqref{eq:op_F} when the following condition holds: if for any $(x,t) \in \Omega $ and any test function $\varphi$ touching $u$ from below at $(x,t)$, then
		\begin{equation*}
			\varphi_t (x,t) - F ( D^2 \varphi (x,t),D \varphi (x,t),x,t ) \ge f(x,t).
		\end{equation*}
	\end{enumerate}
\end{definition}
We introduce the Pucci's extremal operators to define the solutions class.
\begin{definition} [Pucci's extremal operators] \label{def_pucci}
Given elliptictity constants $0<\lambda \leq \Lambda$ and $\mathcal{S}^n \coloneqq \{ M : \text{$M$ is an $n\times n$ real symmetric matrix\}}$, we define \textit{Pucci's extremal operators} as follows:
\begin{align*}
	\mathcal{M}_{\lambda, \Lambda}^{+}(M)  \coloneqq \sup_{\lambda I \leq A \leq \Lambda I} \text{tr} (AM)  	\qquad \text{and} \qquad
	\mathcal{M}^{-}_{\lambda, \Lambda}(M)  \coloneqq \inf_{\lambda I \leq A \leq \Lambda I} \text{tr} (AM).
\end{align*}
\end{definition}

\begin{definition}[Solutions classes]
	Let $f$ be a continuous bounded function in $\Omega \subset \mathbb{R}^n_+\times \mathbb{R}$. We define several solutions classes as follows:
	\begin{align*}
		\underline{S}(\lambda, \Lambda,f ) &\coloneqq  \{u \in C(\Omega)\mid u_t - x_n^{\gamma}\mathcal{M}^{+}_{\lambda, \Lambda}(D^2u) \leq f \ \text{in the viscosity sense}  \}, \\
		\overline{S}(\lambda, \Lambda,f ) &\coloneqq   \{u \in C(\Omega)\mid u_t -x_n^{\gamma} \mathcal{M}^{-}_{\lambda, \Lambda}(D^2u) \ge f  \ \text{in the viscosity sense}   \}, \\
		S(\lambda,\Lambda, f) &\coloneqq \underline{S}(\lambda, \Lambda,f )  \cap \overline{S}(\lambda, \Lambda,f ), \\
		S^*(\lambda,\Lambda, f) &\coloneqq \underline{S}(\lambda, \Lambda, \|f\|_{L^{\infty}(\Omega)})  \cap \overline{S}(\lambda, \Lambda,-\|f\|_{L^{\infty}(\Omega)}).  
	\end{align*}
	We call the functions in $S(\lambda, \Lambda,f )$ solutions. Clearly, we have $S(\lambda, \Lambda, f)$ is contained in $S^{\ast}(\lambda, \Lambda, f)$ and $S^*(\lambda, \Lambda, 0)=S(\lambda, \Lambda, 0)$.
\end{definition}

\begin{remark}
	If $u \in C(\Omega)$ is a viscosity solution of \eqref{eq:model} with ellipticity constants $\lambda$ and $\Lambda$, then we know that $u \in S(\lambda, \Lambda, f(x,t)+x_n^{\gamma}F(O_n, x,t))$ in $\Omega$.  \end{remark}
\subsection{Structure Hypotheses on the operator $F$} \label{sec:shf}
We assume that the fully nonlinear operator $F:\mathcal{S}^n \times \overline{\Omega^+} \to \mathbb{R}$ satisfies some of the following conditions:
\begin{enumerate} [label=\text{(F\arabic*)}]
\item \label{F1} $F$ is uniformly parabolic; that is,  there exist constants $0<\lambda \leq \Lambda$ such that \begin{equation*}
	\mathcal{M}^{-}_{\lambda, \Lambda}(M-N) \leq F(M,x,t) - F(N,x,t) \leq \mathcal{M}^{+}_{\lambda, \Lambda}(M-N)
\end{equation*}
for all $M,N \in \mathcal{S}^n, \, (x,t) \in\overline{\Omega^+}$.
\item \label{F2} $F$ is concave or convex.
\item \label{F3} Let $\alpha \in (0,\overline{\alpha})$ and $(x_0,t_0)\in \overline{\Omega^+}$. There exist a constant $r_0>0$ and two nonnegative functions $\beta^1,\beta^2 \in C^{\alpha}( \overline{\Omega^+})$  such that $\beta^1(x_0,t_0)=0$, $\beta^2 (x_0,t_0)=0$, and
\begin{equation*}
	|F(M,x,t)-F(M,x_0,t_0)| \leq \beta^1(x,t) \|M\| + \beta^2(x,t) 
\end{equation*}
for all $M\in\mathcal{S}^n$, $(x,t) \in \overline{\Omega^+} \cap \overline{Q_{r_0}^+(x_0,t_0)}$.
\item \label{F3'} Let $\alpha \in (0,\overline{\alpha})$ and $(x_0,t_0)\in \overline{\Omega^+}$. There exist a constant $r_0>0$ and two nonnegative functions $\beta^1,\beta^2 \in H^{\alpha}(\overline{\Omega^+})$  such that $\beta^1(x_0,t_0)=0$, $\beta^2 (x_0,t_0)=0$, and
\begin{equation*}
	|F(M,x,t)-F(M,x_0,t_0)| \leq \beta^1(x,t) \|M\| + \beta^2(x,t) 
\end{equation*}
for all $M\in\mathcal{S}^n$, $(x,t) \in \overline{\Omega^+} \cap \overline{Q_{r_0}^+(x_0,t_0)}$.
\item \label{F4} Let $(x_0,t_0)\in \Omega$. There exist constants $r_0>0$, $C_\theta>0$ and two nonnegative functions $\theta$ satisfying $\theta(x_0,t_0)=0$ and 
\begin{equation*}
	\left(\frac{1}{|Q_{r}^+(x_0,t_0)|} \int_{Q_{r}^+(x_0,t_0)} |\theta(x,t)|^{n+1}\,dx\,dt\right)^{1/(n+1)} \leq C_\theta r^{\alpha} \quad \text{for all } r \leq r_0
\end{equation*}
such that
\begin{equation*}
	|F(M,x,t)-F(M,x_0,t_0)| \leq \theta(x,t) (\|M\|+1) 
\end{equation*}
for all $M\in\mathcal{S}^n$, $(x,t) \in \Omega^+ \cap Q_{r_0}^+(x_0,t_0)$. 
\end{enumerate}

\begin{remark}
Below are some comments regarding the structural hypotheses on $F$.
\begin{enumerate}  [label=(\roman*)]
	\item \ref{F1}, \ref{F2}, and \ref{F3} are conditions for boundary $C^{2,\alpha}$-regularity.
	\item \ref{F2} can be reduced to the condition on $F$ such that \eqref{eq:model} has a solution $u \in C_{\textnormal{loc}}^2(\Omega^+)$.
	\item Regarding pointwise regularity, it suffices to have $\beta^1,\beta^2 \in C^{\alpha}( x_0,t_0)$ in  \ref{F3}.
	\item  \ref{F3'} is a condition for global $C^{2,\alpha}$-regularity when $\gamma<0$.
	\item \ref{F4} is a condition for global $C^{1,\alpha}$-regularity. 
	\item \ref{F4} is satisfied if $F(\cdot,x) \in C^{\alpha}(x_0,t_0)$.
\end{enumerate}
\end{remark}
\subsection{Known Regularity Results}
Here we gather some known regularity results that we will need later on. 
\subsubsection{Interior Harnack Inequality}
For a given $\varepsilon \in (0,1)$, the function $u \in S^*(\lambda,\Lambda, f)$ in $\Omega \ssubset \{(x,t) \in Q_1^+ : x_n \ge \varepsilon\}$ satisfies 
\begin{equation*} 
\left\{\begin{aligned} 
	 u_t  -\mathcal{M}^{-}_{\varepsilon^{\gamma} \lambda ,\Lambda}(D^2 u) & \ge  -\|f\|_{L^{\infty}(Q_1^+)} \\
	 u_t  -\mathcal{M}^{+}_{\varepsilon^{\gamma} \lambda ,\Lambda}(D^2 u) & \leq \|f\|_{L^{\infty}(Q_1^+)} 
\end{aligned}\right. \quad \text{in } \Omega
\end{equation*}
when $0 < \gamma <1$ and
\begin{equation*} 
\left\{\begin{aligned} 
	 u_t  -\mathcal{M}^{-}_{\lambda ,\varepsilon^{\gamma}\Lambda}(D^2 u) & \ge  -\|f\|_{L^{\infty}(Q_1^+)} \\
	 u_t  -\mathcal{M}^{+}_{\lambda ,\varepsilon^{\gamma}\Lambda}(D^2 u) & \leq \|f\|_{L^{\infty}(Q_1^+)} 
\end{aligned}\right. \quad \text{in } \Omega
\end{equation*}
when $\gamma \leq 0$. In other words, in a domain sufficiently far away from the boundary $\{(x,t) \in \partial_p Q_1^+:x_n=0\}$, $u$ belongs to the solutions class for certain uniformly parabolic equations, and hence we obtain the following Interior Harnack Inequality.
\begin{theorem} [Interior Harnack inequality, {\cite[Theorem 4.18]{Wan92a}}] \label{int_harnack}
Let $u \in S^*(\lambda, \Lambda,f )$ be a nonnegative function in $Q_1^+$. Suppose that $\Omega\coloneqq B_r(a) 
\times (-3r^{2-\gamma},0]$ is compactly contained in $Q_1^+$ and $f \in C(Q_1^+) \cap L^{n+1}(Q_1^+)$. Then 
\begin{equation*}
	\sup_{B_r(a) \times (-3r^{2-\gamma}, -2r^{2-\gamma}]} u \leq C \left(\inf_{B_r(a) \times (-r^{2-\gamma},0]} u + \| f \|_{L^{n+1}(Q_1^+)} \right),
\end{equation*}
where $C$ is a constant depending only on $n$, $\lambda$, $\Lambda$, $\gamma$, and $d \coloneqq \textnormal{dist}\,(\Omega,\partial_p Q_1^+)$. 
\end{theorem}

\subsubsection{Comparison Principle}
We are interested in the comparison principle of \eqref{eq:model}. \cite[Theorem 8.2]{CIL92} provides a comparison principle for many kinds of fully nonlinear operators, but \eqref{eq:model} does not satisfy   \cite[Condition (3.14)]{CIL92}. To demonstrate the comparison principle of \eqref{eq:model}, we introduce semijets and theorems suggested in \cite[Section 8]{CIL92}.
\begin{definition}[Parabolic semijets]\label{semijet}
	 Let $u$ be a function defined in $Q_1^+$ and let $(x, t)\in Q_1^+$.
	\begin{enumerate}[label=(\roman*)]
		\item A \textit{parabolic superjet} $\mathscr{P}^{2, +}u(x, t)$ consists of $(a, p, M) \in \mathbb{R} \times \mathbb{R}^n \times \mathcal{S}^n$ which satisfy
		\begin{align*}
			u(y, s) &\leq u(x,t)+a(s-t)+p \cdot(y-x) \\
			&\qquad +\frac{1}{2} (y-x)^T M(y-x) +o(|s-t|+|z-x|^2) \quad \text{as $(y, s) \to (x, t)$}. 
		\end{align*}
		Similarly, we can define a \textit{parabolic subjet} $\mathscr{P}^{2, -}u(x, t)$. It immediately follows that
		\begin{align*}
			\mathscr{P}^{2, -}u(x, t)=-\mathscr{P}^{2, +}(-u)(x, t).
		\end{align*} 
	
	\item A \textit{limiting superjet} $\overline{\mathscr{P}}^{2, +}u(x, t)$ consists of $(a, p, M) \in \mathbb{R} \times \mathbb{R}^n \times \mathcal{S}^n$ with the following statement holds: there exists a sequence $\{(x_n, t_n, a_n, p_n, M_n)\}_{n=1}^{\infty}$ such that $(a_n, p_n, M_n) \in \mathscr{P}^{2, +}u(x_n, t_n)$ and 
	\begin{equation*}
		\text{$(x_n, t_n, u(x_n, t_n), a_n, p_n, M_n) \to (x, t, u(x, t), a, p, M)$} \quad \text{as } n \to \infty.
	\end{equation*}
	We define a \textit{parabolic subjet} $\overline{\mathscr{P}}^{2, -}u(x, t)$ in a similar way.
	\end{enumerate}	
\end{definition}
\begin{theorem}[Jensen Ishii's Lemma,  {\cite[Theorem 8.3]{CIL92}}, {\cite[Lemma 2.3.23]{BEG13}}]\label{lem:ishii}
Let $U$ and $V$ be two open sets of $\mathbb{R}^n_+$ and $I$ an open interval of $\mathbb{R}$. Consider also a viscosity subsolution $u$ of \eqref{eq:op_F} in $U\times I$ and a viscosity supersolution $v$ of \eqref{eq:op_F} in $V\times I$. Suppose that 
\begin{equation} \label{ishii_fcn}
	w(x,y,t)\coloneqq u(x,t)-v(y,t) - \frac{1}{2\varepsilon} |x-y|^2
\end{equation} 
has a local maximum at $(x_\varepsilon,y_\varepsilon,t_\varepsilon) \in U\times V \times I$. Then there exists $a \in \mathbb{R}$ and $M,N \in \mathcal{S}^n$ such that $(a,p,M) \in \overline{\mathscr{P}}^{2, +}u(x_\varepsilon, t_\varepsilon)$, $(a,p,N) \in \overline{\mathscr{P}}^{2, -}v(y_\varepsilon, t_\varepsilon)$, and 
\begin{equation*}
	-\frac{2}{\varepsilon} 
		\begin{pmatrix}
			I_n & O_n  \\
			O_n & I_n
		\end{pmatrix} \leq 
		\begin{pmatrix}
			M & O_n  \\
			O_n & -N
		\end{pmatrix} \leq  \frac{3}{\varepsilon} 
		\begin{pmatrix}
			I_n & -I_n  \\
			-I_n & I_n
		\end{pmatrix},
\end{equation*}
where $p \coloneqq \varepsilon^{-1}(x_\varepsilon-y_\varepsilon)$, $I_{n}$ and $O_{n}$ are the indentity matrix and zero matrix, respectively.
\end{theorem} 
\begin{lemma} [{\cite[Lemma 3.1]{CIL92}}, {\cite[Lemma 2.3.19]{BEG13}}] \label{lem:mu_ep}
	Let $u$ be a upper semicontinuous in $\overline{Q_1^+}$ and $v$ be a lower semicontinuous in $\overline{Q_1^+}$. Assume that 
	\begin{equation*}
		\mu_{\varepsilon} \coloneqq \sup_{(x,t),(y,s) \in \overline{Q_1^+}} \left(u(x,t)-v(y,s) -\frac{1}{2\varepsilon}|x-y|^2   -\frac{1}{2\varepsilon}|t-s|^2 \right) < \infty
	\end{equation*}
	for small $\varepsilon>0$ and $\{(x_\varepsilon, t_\varepsilon, y_\varepsilon,s_\varepsilon)\}$ satisfies that 
	\begin{equation*}
		\lim_{\varepsilon \to 0} \left( \mu_{\varepsilon} - u(x_\varepsilon,t_\varepsilon)+v(y_\varepsilon,s_\varepsilon) + \frac{1}{2\varepsilon}|x_\varepsilon-y_\varepsilon|^2 + \frac{1}{2\varepsilon}|t_\varepsilon-s_\varepsilon|^2\right)=0.
	\end{equation*}
	Then the following statements hold:
	\begin{enumerate}[label=(\roman*)]
	\item $\displaystyle\lim_{\varepsilon \to 0} {\varepsilon}^{-1} |x_\varepsilon - y_\varepsilon|^2 =0$ and $\displaystyle\lim_{\varepsilon \to 0} {\varepsilon}^{-1} |t_\varepsilon - s_\varepsilon|^2 =0$,
	\item $\displaystyle\lim_{\varepsilon \to 0} \mu_{\varepsilon} = u(\overline{x},\overline{t})-v(\overline{x},\overline{t}) = \sup_{Q_1^+}(u-v)$ whenever $(\overline{x},\overline{t}) \in \overline{Q_1^+}$ is a limit point of $\{(x_\varepsilon, t_\varepsilon)\}$ as $\varepsilon \to 0$.
	\end{enumerate}
\end{lemma} 
We are now ready to prove the comparison principle for \eqref{eq:model}.
\begin{theorem}[Comparison principle]\label{com_prin}
Let $u \in C(\overline{Q_1^+})$ and $v\in C(\overline{Q_1^+})$ be a viscosity subsolution and viscosity supersolution of \eqref{eq:model} in $Q_1^+$, respectively. If $u \leq v$ on $\partial_p Q_1^+$, then $u \leq v$ in $Q_1^+$.
\end{theorem} 

\begin{proof}
	We first observe that for $\delta>0$, $\tilde{u}=u + \delta/t$ is also viscosity subsolution of \eqref{eq:model}. Since $u \leq v$ in $Q_1^+$ follows from $\tilde{u} \leq v $ in $Q_1^+$ and letting $\delta \to 0$, we may assume that
	\begin{equation*}
		u_t - x_n^{\gamma} F(D^2 u,x,t) \leq  f - \delta \quad \text{in the viscosity sense}
	\end{equation*}
	and $u \to -\infty$ uniformly on $\overline{\{x\in \mathbb{R}_+^n: |x_i| \leq 1 \,(1\leq i \leq n)\}}$ as $t \to 0$.
	
	We argue by contradiction. Suppose that $m \coloneqq u(x_0,t_0)-v(x_0,t_0) = \sup_{Q_1^+}(u-v) >0$ for some $(x_0,t_0) \in Q_1^+$. Since $u$ is bounded above, $w$ defined as in \eqref{ishii_fcn} has the maximum $\mu_\varepsilon$  at $(x_\varepsilon,y_\varepsilon,t_\varepsilon)$. Furthermore, $\{(x_\varepsilon,y_\varepsilon)\}$ is a bounded sequence, there exists a subsequence $\{(x_{\varepsilon_j},y_{\varepsilon_j})\}$ such that $x_{\varepsilon_j} \to \overline{x}$ as $j \to \infty$ and $y_{\varepsilon_j} \to \overline{y}$ as $j \to \infty$. If $t_\varepsilon = -1$, by \Cref{lem:mu_ep}, we have $\overline{x} =\overline{y}$ and 
	\begin{align*}
		0<m \leq \lim_{j \to \infty} \mu_{\varepsilon_j} &\leq \lim_{\varepsilon_j \to 0}\sup_{(x,-1),(y,-1) \in Q_1^+} \left(u(x,-1)-v(y,-1) -\frac{1}{2\varepsilon_j}|x-y|^2  \right) \\
		&=u(\overline{x},-1) -y(\overline{x},-1) \leq \sup_{\partial_p Q_1^+}(u-v) \leq 0
	\end{align*}
	which is not possible. Similarly, we can see that $(x_\varepsilon,t_\varepsilon), (y_\varepsilon,t_\varepsilon) \in Q_1^+$ if $\varepsilon>0$ is sufficiently small. Thus we can apply \Cref{lem:ishii} to the point $(x_\varepsilon,y_\varepsilon,t_\varepsilon)$, there exists $a \in \mathbb{R}$ and $M,N \in \mathcal{S}^n$ such that $(a,p,M) \in \overline{\mathscr{P}}^{2, +}u(x_\varepsilon, t_\varepsilon)$, $(a,p,N) \in \overline{\mathscr{P}}^{2, -}v(y_\varepsilon, t_\varepsilon)$, and 
$M \leq N$, where $p \coloneqq \varepsilon^{-1}(x_\varepsilon-y_\varepsilon)$. These imply that

\begin{equation*}
	a -x_{\varepsilon n}^{\gamma} F(M,x_\varepsilon, t_\varepsilon) \leq f(x_\varepsilon, t_\varepsilon) -\delta \quad \text{and} \quad
	a -y_{\varepsilon n}^{\gamma} F(N,x_\varepsilon, t_\varepsilon) \ge f(y_\varepsilon, t_\varepsilon) , 
\end{equation*}
where $x_{\varepsilon n}$ is the $n$-th coordinate of $x_\varepsilon$. Since $\lim_{\varepsilon \to 0}\varepsilon^{-1} |x_\varepsilon - y_\varepsilon|^2 = 0$ and $M \leq N$, we have
\begin{align*}
	\delta &\leq  f(x_\varepsilon, t_\varepsilon) - f(y_\varepsilon, t_\varepsilon) + x_{\varepsilon n}^{\gamma} F(M,x_\varepsilon, t_\varepsilon)- y_{\varepsilon n}^{\gamma} F(N,x_\varepsilon, t_\varepsilon) \\
	&\leq f(x_\varepsilon, t_\varepsilon) - f(y_\varepsilon, t_\varepsilon) + (x_{\varepsilon n}^{\gamma} - y_{\varepsilon n}^{\gamma}) F(N,x_\varepsilon, t_\varepsilon) \to 0 \quad \text{as } \varepsilon \to 0
\end{align*}
which leads to a contradiction.
\end{proof}
\subsubsection{Existence of Viscosity Solutions}
If $u$ is bounded from below in $Q_1^+$, one can define the \textit{lower semicontinuous envelope} $u_*$ of $u$ as the largest lower semicontinuous function lying below $u$. Similarly, the \textit{upper semicontinuous envelope} $u^*$ of $u$ can be defined if $u$ is bounded from above in $Q_1^+$.
\begin{lemma} [{\cite[Lemma 2.3.15]{BEG13}}] \label{lem:perron}
	Let $u^- \in C(\overline{Q_1^+})$ and $u^+\in C(\overline{Q_1^+})$ be a viscosity subsolution and viscosity supersolution of \eqref{eq:model} in $Q_1^+$, respectively. Then there exists a function $u:Q_1^+ \to \mathbb{R}$ such that $u^-\leq u \leq u^+$, $u^*$ is a subsolution of \eqref{eq:model} in $Q_1^+$, and $u_*$ is a supersolution of \eqref{eq:model} in $Q_1^+$.
\end{lemma}
\begin{lemma} [Existence of the viscosity solutions]
Let $f \in C(Q_1^+) \cap L^{\infty}(Q_1^+)$ and $g \in C(\partial_p Q_1^+)$. Then the Cauchy-Dirichlet problem
\begin{equation} \label{eq:diri_model}
	\left\{\begin{aligned}
		u_t - x_n^{\gamma} F(D^2 u,x,t) &= f && \text{in } Q_1^+\\
		u &= g && \text{on } \partial_p Q_1^+
	\end{aligned} \right.
\end{equation}
has a unique viscosity solution.
\end{lemma}

\begin{proof}
Consider functions 
\begin{equation*}
	u^{\pm}=\pm e^{ t + 1}   ( \|f\|_{L^{\infty}(Q_1^+)} +  \|g\|_{C(\partial_p Q_1^+)}  ) .
\end{equation*}
Note that $u^+$ is a viscosity supersolution of \eqref{eq:diri_model} and $u^{-} $ is a viscosity subsolution of \eqref{eq:diri_model}. Then, by \Cref{lem:perron}, there exists a function $u:Q_1^+ \to \mathbb{R}$ such that $u^-\leq u \leq u^+$, $u^*$ is a viscosity subsolution of \eqref{eq:diri_model} in $Q_1^+$, and $u_*$ is a viscosity supersolution of \eqref{eq:diri_model} in $Q_1^+$. Since $u^*$ is viscosity subsolution of \eqref{eq:diri_model} in $Q_1^+$ and $u_*$ is viscosity supersolution of \eqref{eq:diri_model} in $Q_1^+$, by the comparison principle [\Cref{com_prin}], we have $u^* \leq u_*$ in $Q_1^+$. Furthermore, $u_*\leq u \leq u^*$ by definition of the semicontinuous envelopes, we have $u = u_* =u^*$ and hene $u$ is a viscosity solution of \eqref{eq:diri_model}. The uniqueness of the solution is obtained directly by the comparison principle.
\end{proof}

%
%
\section{Boundary H\"older Gradient Estimates} \label{sec:c1a}
In this section, we prove the boundary $C^{1,\alpha}$-regularity for functions in $S^*(\lambda,\Lambda, f)$. We first prove Lipschitz estimate and Hopf principle. These are the key lemmas for proving the boundary $C^{1,\alpha}$-regularity for functions in $S^*(\lambda,\Lambda, f)$.

\subsection{Lipschitz Estimate and Hopf Principle}
Lipschitz estimation and Hopf principle are mainly concerned with the construction of barrier functions. However, because of term $x_n^{\gamma}$, the construction of the barrier function using the fundamental solution of the heat equation is not suitable. That is, it is necessary to construct a barrier function considering the influence of $x_n^{\gamma}$ and consider a suitable domain for it.

To simplify the notation in the proof, we denote $\{e_1,\cdots, e_n\}$ the standard basis of $\mathbb{R}^n$.
\begin{lemma} [Lipschitz estimate] \label{lem:u<x}
Let $f \in C(Q_1^+) \cap L^{\infty}(Q_1^+)$. Assume that $u \in C(\overline{Q_1^+})$ satisfies
\begin{equation*}
	\left\{\begin{aligned}
		u&\in S^*(\lambda,\Lambda, f) && \text{in } Q_1^+ \\
		u&=0 && \text{on } \partial_p Q_1^+\cap \{x_n=0\}.
	\end{aligned}\right.
\end{equation*}
Then 
\begin{equation*}
	|u(x,t)|\leq C ( \|u\|_{L^{\infty}(Q_1^+)} + \|f\|_{L^{\infty}(Q_1^+)}) |x| \quad \text{for all } (x,t) \in \overline{Q_{1/2}^{+}},
\end{equation*}
where $C>0$ is a constant depending only on $n$, $\lambda$, $\Lambda$, and $\gamma$.
\end{lemma}

\begin{proof}
For any $\delta > 0$, replacing $u$ by 
$$ \frac{u}{\delta+\|u\|_{L^{\infty}(Q_1^+)} + \|f\|_{L^{\infty}(Q_1^+)} }$$ 
and letting $\delta \to 0^+$, we may assume that 
\begin{equation*}
	\|u\|_{L^{\infty}(Q_1^+)} \leq 1 \qquad  \mbox{and} \qquad  \|f\|_{L^{\infty}(Q_1^+)} \leq 1.
\end{equation*}

Let $\eta(x)=1- |x+e_n|^{-\beta}$ for the constant $\beta>\max\left\{2, (n-1)\Lambda/\lambda-1\right\}$ and consider following functions
\begin{equation*}
\tilde{u}= (t+1)u \qquad \mbox{and} \qquad v = 2M \eta - Mx_n^{2-\gamma}.
\end{equation*}
Then, it can be seen that $\tilde{u} \leq v$ on $\partial_p Q_1^{+}$ for sufficiently large $M>0$.

We claim that $\tilde{u}$ and $v$ are a viscosity subsolution and supersolution respectively of the following equation: 
\begin{equation*}
	w_t - x_n^{\gamma} \mathcal{M}_{\lambda,\Lambda}^{+} ( D^2 w) = 2 \quad \text{in } Q_1^+.
\end{equation*}
For any $(x,t) \in Q_1^+$, let $\tilde{\varphi}$ be a test function that touching $\tilde{u}$ from above at $(x,t)$. Then the test function $\varphi \coloneqq \tilde{\varphi}/(t+1)$ touches $u$ from above at $(x,t)$. Since $u \in S^*(\lambda, \Lambda, f)$ in $Q_1^+$, we know that
$$\varphi_t (x,t) - x_n^{\gamma} \mathcal{M}_{\lambda,\Lambda}^{+}  ( D^2 \varphi (x,t) ) \leq  \|f\|_{L^{\infty}(Q_1^+)} $$
and hence we have
\begin{align*}
	\tilde{\varphi}_t (x,t) - x_n^{\gamma} \mathcal{M}_{\lambda,\Lambda}^{+} \left(D^2 \tilde{\varphi}(x,t) \right) \leq  (t+1)\|f\|_{L^{\infty}(Q_1^+)} +  u (x,t)  \leq 2. 
\end{align*}

On the other hand, we can see that
\begin{equation*}
	\begin{aligned}
		D^2\eta(x) &\sim \beta|x+e_n|^{-\beta-2} 
			\left(\begin{array}{c|c}
				-(1+\beta) & \textbf{0} \\ \hline
				\textbf{0} & I_{n-1}
			\end{array}\right), \\ 
		D^2 x_n^{2-\gamma}  &= (2-\gamma)(1-\gamma)x_n^{-\gamma} 
			\left(\begin{array}{c|c}
				O_{n-1} & \textbf{0} \\ \hline
				\textbf{0} & 1
			\end{array}\right)
	\end{aligned}
\end{equation*}
by rotational symmetry, where $I_{n-1}$ and $O_{n-1}$ are the indentity matrix and zero matrix. Then, we have
\begin{align*}
	v_t - x_n^{\gamma} \mathcal{M}_{\lambda,\Lambda}^{+}( D^2 v) &= - M x_n^{\gamma} \mathcal{M}_{\lambda,\Lambda}^{+}( 2 D^2 \eta -D^2 x_n^{2-\gamma}) \\
	& \ge - 2 M x_n^{\gamma} \mathcal{M}_{\lambda,\Lambda}^{+}( D^2 \eta) + M x_n^{\gamma} \mathcal{M}_{\lambda,\Lambda}^{-}(D^2 x_n^{2-\gamma}) \\
&= 2 \beta M x_n^{\gamma} |x+e_n|^{-\beta-2} \big((1+\beta)\lambda- (n-1)\Lambda \big) + \lambda M (2-\gamma)(1-\gamma) \\
&\ge 2
\end{align*}
for sufficiently large $M>0$. By the comparison principle [\Cref{com_prin}], we have
$$(t+1) u \leq  2M \eta - Mx_n^{2-\gamma} \quad \mbox{in } Q_1^+.$$
Therefore for $0<r<1$, we conclude that
\begin{align*}
	u(x,t)&\leq \frac{1}{t+1}\big( 2M\eta -Mx_n^{2-\gamma} \big) \leq \frac{2M}{1-r^{2-\gamma}}\left(1-\frac{1}{(1+|x|)^{\beta}}\right)\leq\frac{2\beta M}{1-r^{2-\gamma}} |x|
\end{align*}
for all $(x,t) \in \overline{Q_r^+}$. Similarly, we can obtain the lower bound by considering the following functions
\begin{equation*}
\hat{u}=-(t+1)u \qquad \mbox{and} \qquad v = 2M\eta - Mx_n^{2-\gamma}.
\end{equation*}
\end{proof}

\begin{remark} \label{rmk_trans_inv}
	Assume that $u \in S^*(\lambda, \Lambda, f)$ in $Q_1^+$ and let $v(x,t)=u(x-y,t)$ be a translation in direction  $y \perp e_n$. Then $v\in S^*(\lambda, \Lambda, \tilde{f})$ in $Q_r^+$ for some $r \in (0,1)$, where $\tilde{f}(x,t) \coloneqq f(x-y,t)$. We can apply \Cref{lem:u<x} to the function $v$ and hence we have
$$|u(x,t)| \leq C \big( \|u\|_{L^{\infty}(Q_1^+)} + \|f\|_{L^{\infty}(Q_1^+)} \big)   x_n \quad \text{for all }  (x,t) \in \overline{Q_{r/2}^+}.$$
\end{remark}
\begin{lemma} [Hopf principle] \label{hopf}
Let $u \in C(\overline{Q_1^+})$ be a nonnegative function satisfying 
\begin{equation*}
	\left\{\begin{aligned}
		u&\in S(\lambda,\Lambda, 0) && \text{in } Q_1^+ \\
		u&=0 && \text{on } \partial_p Q_1^+\cap \{x_n=0\}.
	\end{aligned}\right.
\end{equation*}
Then  
 \begin{equation} \label{hopf_conclusion}
 	u(x,t) \ge C u(e_n/2,-2/4^{2-\gamma}) x_n \quad \mbox{for all } (x,t) \in \overline{Q_{1/2}^+},
 \end{equation}
where $C$ is a constant depending only on $n$, $\lambda$, $\Lambda$, and $\gamma$.
\end{lemma}

\begin{proof}
By the interior Harnack inequality [\Cref{int_harnack}], we have 
\begin{equation*}
	u(e_n/2,-2/4^{2-\gamma})\leq \sup_{B_{1/4} (e_n/2)\times(-3/4^{2-\gamma},-2/4^{2-\gamma}]} u \leq C_0^{-1}  \inf_{B_{1/4}(e_n/2) \times (-1/4^{2-\gamma},0]} u,
\end{equation*}
where $C_0>0$ is a constant depending only on $n$, $\lambda$, $\Lambda$, and $\gamma$. Then it immediately follows that
\begin{equation*}
	u \ge C_0 u(e_n/2,-2/4^{2-\gamma}) \quad \mbox{on } \overline{B_{1/4}(e_n/2) \times (-1/4^{2-\gamma},0]}.
\end{equation*}

First, let us consider the case of singular equations, i.e., $\gamma < 0$. In this case, consider functions $v, \varphi: \Omega \to \mathbb{R}$ given by
\begin{equation*}
	v(x,t) \coloneqq  C_0 u(e_n/2,-2/4^{2-\gamma})   \frac{ \varphi(x,t)- e^{-\beta/4}}{e^{-\beta/16 }- e^{-\beta/4}}  
	\qquad \text{and} \qquad
	\varphi(x,t) \coloneqq e^{-\beta ( |x-e_n/2|^2 -4^{1-\gamma} t)} ,
\end{equation*}
where $\Omega = \{(x,t): 1/4 < |x-e_n/2| < \sqrt{4^{1-\gamma}t+ 1/4}, \, t < 0 \} $. Since the parabolic boundary is 
\begin{align*}
	\partial_p \Omega&= S_1 \cup S_2, \\ 
	S_1 &\coloneqq \{(x,t) \in \overline{\Omega} : |x-e_n/2| = \sqrt{ 4^{1-\gamma}t + 1/4 }, \, t < 0 \},  \\
	S_2 &\coloneqq \{(x,t) \in \overline{\Omega}:  |x-e_n/2| =1/4, \, t < 0 \} ,
\end{align*}
we have $u \ge v$ on $\partial_p \Omega$.

For any matrix $A= (a^{ij}) \in \mathcal{S}^n$ with $\lambda I \leq A \leq \Lambda I$ and sufficiently large $\beta >0$, we have
\begin{align*}
	(x_n^{-\gamma} \varphi_t -  a^{ij} \varphi_{ij} ) / \varphi 
	& \leq 4^{1-\gamma} \beta - \frac{1}{4} \lambda \beta^2 + 2 n \Lambda \beta  < 0  \quad \text{in } \Omega
\end{align*}
and hence
$$\varphi_t  - x_n^{\gamma} \mathcal{M}_{\lambda,\Lambda}^{-}(D^2 \varphi )  = \sup_{\lambda I \leq A \leq \Lambda I}  (\varphi_t -  x_n^{\gamma} a^{ij} \varphi_{ij}) \leq 0 \quad \mbox{in } \Omega.$$
This implies that 
\begin{equation*}
	v_t  -  x_n^{\gamma} \mathcal{M}_{\lambda,\Lambda}^{-}(D^2 v) \leq 0\quad \mbox{in } \Omega. 
\end{equation*}
By the comparison principle [\Cref{com_prin}], we have
\begin{equation*}
	u(0',x_n,0) \ge v(0',x_n,0) \ge C  u(e_n/2,-2/4^{2-\gamma}) x_n \quad \text{for }  0<x_n<1/4.
\end{equation*}
For the same reason as \Cref{rmk_trans_inv}, we know that \eqref{hopf_conclusion} holds by considering the translation in the $x'$ and $t$ directions.

Next, let us consider the case of degenerate equations, i.e., $0<\gamma<1$. Consider functions $w: \Sigma \to \mathbb{R}$ given by
\begin{align*}
	w(x,t) & \coloneqq C_0 u(e_n/2,-2/4^{2-\gamma})  e^{\beta(\delta-2\cdot 4^{2-\gamma})/4^{3-\gamma}}  \frac{ e^{-\beta  ( |x-e_n/2|^2- \delta t x_n -1/4 )} - e^{-\beta(4^{2-\gamma}-\delta)tx_n}}{e^{-\beta/16 }- e^{-\beta/4}},
\end{align*}
where $\Sigma = \{ (x,t) : 1/4< |x-e_n/2| < \sqrt{4^{2-\gamma}t x_n +1/4}, \, t < 0 \}$. Since the parabolic boundary is 
\begin{align*}
	\partial_p \Sigma &= T_1 \cup T_2, \\
	T_1 & \coloneqq \{X \in \overline{\Sigma} : |x-e_n/2| = \sqrt{4^{2-\gamma}t x_n + 1/4}, \, t < 0 \},  \\
	T_2 & \coloneqq \{X \in \overline{\Sigma}:  |x-e_n/2| =1/4, \, t < 0 \},
\end{align*}
we have $u \ge w$ on $\partial_p \Sigma$.
Note that for any $(x,t) \in \Sigma$, the following inequalities hold
\begin{equation*}
	\frac{1}{4} < |x-e_n/2| < \frac{1}{2} ,\quad  -\frac{1}{4^{2-\gamma}}< t \leq 0 \quad \text{and} \quad -tx_n < \frac{1}{4^{3-\gamma}}.
\end{equation*}

Let $	\xi(x,t)  \coloneqq e^{-\beta  ( |x-e_n/2|^2- \delta t x_n -1/4 )}$. Then for $\delta <\frac{\lambda}{20\Lambda}$ and $\beta > \frac{30(\delta +2n\Lambda)}{\lambda}$, we have
\begin{align*}
	(x_n^{-\gamma} \xi_t -  a^{ij} \xi_{ij})/\xi
	& \leq \beta \delta - 4 \beta^2\lambda |x -e_n/2 |^2 + 2 n \beta \Lambda + 4\beta^2 \delta \Lambda |x - e_n/2| |t| - \beta^2 \delta^2 \lambda t^2  \\
	& \leq \beta \delta - \frac{1}{4} \beta^2\lambda + 2 n \beta \Lambda +2\beta^2\delta\Lambda \\
	&\leq  \beta \delta - \frac{1}{5} \beta^2\lambda + 2 n \beta \Lambda  \leq   - \frac{1}{6} \beta^2\lambda \quad \text{in } \Sigma. 
\end{align*}
It follows that
\begin{equation} \label{op_xi}
	x_n^{-\gamma} \xi_t -  a^{ij} \xi_{ij}  \leq - \frac{1}{6} \beta^2\lambda e^{-\beta  ( |x-e_n/2|^2- \delta t x_n -1/4 )} \quad \text{in } \Sigma.
\end{equation}

Let $\psi(x,t) \coloneqq  - e^{-\beta(4^{2-\gamma}-\delta)tx_n}$. Then for $\beta > 4^{2-\gamma}/\Lambda$, we have
\begin{align}
	x_n^{-\gamma} \psi_t -  a^{ij}\psi_{ij}& = \beta(4^{2-\gamma}-\delta) x_n^{1-\gamma} e^{-\beta(4^{2-\gamma}-\delta)tx_n} + a^{nn} \beta^2 (4^{2-\gamma}-\delta)^2 t^2 e^{-\beta(4^{2-\gamma}-\delta)tx_n} \nonumber\\
	&\leq  (4^{2-\gamma}\beta + \beta^2 \Lambda) e^{-\beta(4^{2-\gamma}-\delta)tx_n}\nonumber \\
	&\leq 2\beta^2 \Lambda  e^{-\beta(16-\delta)tx_n} \quad \text{in } \Sigma. \label{op_psi}
\end{align}

Combining \eqref{op_xi} and \eqref{op_psi} gives 
\begin{align*}  
	&x_n^{-\gamma} w_t -  a^{ij} w_{ij} \nonumber \\
	&\quad  \leq \frac{C_0 u(e_n/2,-2/4^{2-\gamma})  e^{\beta(\delta-2\cdot 4^{2-\gamma})/4^{3-\gamma}} }{e^{-\beta/16 }- e^{-\beta/4}} \beta^2 \Big(2 \Lambda  e^{-\beta(4^{2-\gamma}-\delta)tx_n} - \frac{1}{6} \lambda e^{-\beta  ( |x-e_n/2|^2- \delta t x_n -1/4 )}\Big) 
\end{align*}
in $\Sigma$. Since $ |x-e_n/2| < \sqrt{4^{2-\gamma}t x_n +1/4}$, we can see that 
\begin{equation*}
	-\beta(4^{2-\gamma}-\delta)tx_n < -\beta  ( |x-e_n/2|^2- \delta t x_n -1/4 )
\end{equation*}
Then for sufficiently large $\beta>0$, we have $w_t -x_n^{\gamma}a^{ij}w_{ij}  \leq 0$ in $\Sigma$. Thus, we obtain the desired conclusion as in the case of singular equations.
\end{proof}
\subsection{Boundary $C^{1,\alpha}$-estimate for functions in $S^*(\lambda,\Lambda, f)$.}
We will first apply the iteration procedure to the simplified problems and use the compactness argument in the general case.
\begin{lemma} \label{lem:c1a_f0_g0}
	Let $u \in C(\overline{Q_1^+})$ satisfy 
\begin{equation*}
	\left\{\begin{aligned}
		u&\in S(\lambda,\Lambda, 0) && \text{in } Q_1^+ \\
		u&=0 && \text{on } \partial_p Q_1^+\cap \{x_n=0\}.
	\end{aligned}\right.
\end{equation*}
Then there exists $\overline{\alpha}\in(0,1)$ depending only on $n$, $\lambda$, $\Lambda$, and $\gamma$ such that $u \in C^{1,\overline{\alpha}}(0,0)$, i.e., there exists a constant $a \in \mathbb{R}$ such that
\begin{equation*}
	|u(x,t)-ax_n| \leq C_1\|u\|_{L^{\infty}(Q_1^+)}(|x| + |t|^{\frac{1}{2-\gamma}})^{\overline{\alpha}} x_n \quad \text{for all } (x,t) \in \overline{Q_{1/2}^+}
\end{equation*}
and $|a| \leq C_1$, where $C_1>1$ is a constant depending only on $n$, $\lambda$, $\Lambda$, $\gamma$, and $\overline{\alpha}$.
\end{lemma}

\begin{proof}
For any $\delta > 0$, replacing $u$ by $u/(\delta+\|u\|_{L^{\infty}(Q_1^+)})$ and letting $\delta \to 0^+$, we may assume that $\|u\|_{L^{\infty}(Q_1^+)} \leq 1$. We will construct an increasing sequence $\{a_k\}_{k=0}^{\infty}$ and a decreasing sequence $\{b_k\}_{k=0}^{\infty}$ such that for any $k \ge 1$, we have
\begin{equation} \label{induction}
	\left\{ \begin{aligned}
		&a_k x_n \leq u(x,t) \leq b_k x_n \quad \text{for all } (x,t) \in \overline{Q_{2^{-k}}^+}  \\
		&0 \leq b_k - a_k \leq \mu 2^{-k\overline{\alpha}},  \\
	\end{aligned}\right.
\end{equation}
where $\mu >0$ and $\overline{\alpha} \in (0,1)$ are constants which will be determined later.

We argue by induction. The first step in this construction is obtained by Lipschitz estimate. By \Cref{lem:u<x}, we have
\begin{equation*}
	-\overline{C}x_n \leq u(x,t) \leq \overline{C}x_n \quad \text{for all } (x,t) \in \overline{Q_{1/2}^{+}},
\end{equation*}
where $\overline{C}>0$ is determined in \Cref{lem:u<x}. In this case, we can take $a_0= -2\overline{C}$, $a_1= - \overline{C}$, $b_0= 2\overline{C}$, $b_1= \overline{C}$, and $\mu = 2^{\overline{\alpha}+1} \overline{C}$, then \eqref{induction} holds for $k=1$.

Suppose that \eqref{induction} holds for $k \geq 1$. We claim that \eqref{induction} also holds for $(k+1)$. Let $r_k=2^{-k}$. Based on the induction assumption, there are two possible cases:
\begin{equation*} 
	\text{either} \quad u(r_ke_n/2,-2r_k^{2-\gamma}/4^{2-\gamma}) \ge \dfrac{r_k}{4}(a_k+b_k) \quad \text{or} \quad
	u(r_ke_n/2,-2r_k^{2-\gamma}/4^{2-\gamma}) < \dfrac{r_k}{4}(a_k+b_k) .
\end{equation*}
Let us first assume the former. Then the function 
\begin{equation*}
	w(x,t)=\frac{u(r_kx,r_k^{2-\gamma}t)-a_k r_k x_n}{r_k^{1+\overline{\alpha}}} \in [0,\mu]   
\end{equation*}
satisfies
\begin{equation*}
	\left\{\begin{aligned}
		w&\in S(\lambda,\Lambda, 0) && \text{in } Q_1^+ \\
		w&= 0 && \text{on } \partial_p Q_1^+\cap \{x_n=0\}.
	\end{aligned}\right.
\end{equation*}
Note that if $b_k -a_k \leq \mu r_{k+1}^{\overline{\alpha}}$, then we take $a_{k+1} = a_k$ and $b_{k+1} = b_k$. So we can assume that $b_k -a_k > \mu r_{k+1}^{\overline{\alpha}}$. By the Hopf principle [\Cref{hopf}], we have 
 \begin{equation*}
 	w(x,t) \ge \widetilde{C} w(e_n/2,-2/4^{2-\gamma}) x_n \ge \frac{1}{4r_k^{\overline{\alpha}}}\widetilde{C} (b_k-a_k) x_n  >  2^{-\overline{\alpha}-2}  \widetilde{C} \mu x_n \quad \text{for all } (x,t) \in \overline{Q_{1/2}^+},
\end{equation*}
where $\widetilde{C}>0$ is determined in \Cref{hopf}. It follows that
 \begin{equation*}
 	u(x,t) \ge \Big(a_k + 2^{-\overline{\alpha}-2}\widetilde{C}\mu r_k^{\overline{\alpha}} \Big) x_n  \quad \text{for all } (x,t) \in\overline{Q_{r_{k+1}}^+}.
\end{equation*}
We now take 
\begin{equation*}
	a_{k+1} = a_k +  2^{-\overline{\alpha}-2}\widetilde{C}\mu r_k^{\overline{\alpha}} \quad \text{and} \quad  b_{k+1}=b_k,   
\end{equation*}
and the constant $\overline{\alpha} \in (0,1)$ is chosen such that $2^{\overline{\alpha}} \leq 1+ \tilde{C}/4$. Then we have
\begin{align*}
	b_{k+1}-a_{k+1} =b_k -a_k -2^{-\overline{\alpha}-2}\widetilde{C}\mu r_k^{\overline{\alpha}}\leq (1 -2^{-\overline{\alpha}-2}\widetilde{C})\mu r_k^{\overline{\alpha}} \leq \mu r_{k+1}^{\overline{\alpha}}. 
\end{align*}
Thus, \eqref{induction} holds for $(k+1)$. Similarly, we can obtain the iterative step in the case of $u(r_ke_n/2,-2r_k^{2-\gamma}/4^{2-\gamma}) < r_k (a_k+b_k)/4$.

Finally, this construction ensures that there exists a constant $a \in \mathbb{R}$ such that
\begin{equation*}
	\lim_{k\to\infty} a_k=a=\lim_{k\to\infty} b_k \quad \text{and} \quad  |a| \leq 2\overline{C}. 
\end{equation*}
Furthermore,  for any $(x,t) \in \overline{Q_{1/2}^+}$, we can choose $k\in\mathbb{N}$ such that $r_{k+1} < \max\{|x|,|t|^{\frac{1}{2-\gamma}}\} \leq r_k$, and hence we have
\begin{equation*}
	u(x,t)-ax_n \leq (b_k-a)x_n \leq 2^{\overline{\alpha}} \mu r_{k+1}^{\overline{\alpha}} x_n \leq C (|x|+|t|^{\frac{1}{2-\gamma}})^{\overline{\alpha}} x_n.
\end{equation*}
Similarly, we can obtain the lower bound.
\end{proof}
\begin{lemma} \label{lem1:gen}
	Let $\overline{\alpha}\in (0,1)$ and $C_1>1$ be as in \Cref{lem:c1a_f0_g0}. For any $\alpha \in (0,\overline{\alpha})$, there exists $\delta >0$ depending only on $n$, $\lambda$, $\Lambda$, $\gamma$, and $\alpha$ such that if $u \in C(\overline{Q_1^+})$ satisfies
\begin{equation*}
	\left\{\begin{aligned}
		u&\in S^*(\lambda,\Lambda, f) && \text{in } Q_1^+ \\
		u&=g && \text{on } \partial_p Q_1^+\cap \{x_n=0\}
	\end{aligned}\right.
\end{equation*}
with $\|u\|_{L^{\infty}(Q_1^+)} \leq 1$, $\|f\|_{L^{\infty}(Q_1^+)} \leq \delta$, and $\|g\|_{L^{\infty}(\partial_p Q_1^+\cap \{x_n=0\})} \leq \delta$, then there exists a constant $a \in \mathbb{R}$ such that
\begin{equation*}
	\|u-ax_n\|_{L^{\infty}(Q_{\eta}^+)} \leq \eta^{1+\alpha}  
\end{equation*}
and $|a| \leq C_1$, where $\eta>0$ is a constant depending only on $n$, $\lambda$, $\Lambda$, $\gamma$, and $\alpha$.
\end{lemma}

\begin{proof}
We prove by contradiction; we suppose the conclusion does not hold. In other words, there exist sequences $u_k$, $f_k$, and $g_k$ such that $u_k$ satisfies
	\begin{equation*}
		\left\{\begin{aligned}
			u_k & \in S^*(\lambda,\Lambda, f_k) && \text{in } Q_1^+ \\
			u_k &= g_k && \text{on } \partial_p Q_1^+ \cap \{x_n=0\},
		\end{aligned}\right.
	\end{equation*}
$\|u_k\|_{L^{\infty}(Q_1^+)} \leq 1$, $\|f_k\|_{L^{\infty}(Q_1^+)} \leq 1/k$, and $\|g_k\|_{L^{\infty}(\partial_p Q_1^+\cap \{x_n=0\})} \leq 1/k$. Moreover, for any constant $a \in \mathbb{R}$ satisfying $|a| \leq C_1$, we have
\begin{align}\label{eq:contradiction}
	\|u_k-ax_n\|_{L^{\infty}(Q_{\eta}^+)} >\eta^{1+\alpha},
\end{align}
where $C_1>1$ and $\eta \in (0,1)$ will be determined later.

We first note that $\{u_k\}_{k=1}^{\infty}$ is uniformly bounded in $L^{\infty}(Q_1^+)$ and by the interior H\"older estimate (\cite[Theorem 4.19]{Wan92a}), $\{u_k\}_{k=1}^{\infty}$ is equicontinuous in compact sets of $Q_1^+$. Thus, by Arzela-Ascoli theorem, there exist a subsequence $\{u_{k_j}\}_{j=1}^{\infty}$ of $\{u_k\}_{k=1}^{\infty}$ and a limit function $\overline{u} \in C(Q_1^+)$ such that $u_{k_j} \to \overline{u}$ uniformly in compact sets of $Q_1^+$ as $j \to \infty$. Moreover, since $\|f_{k_j}\|_{L^{\infty}(Q_1^+)} \to 0$ as $j \to \infty$ and $\|g_{k_j}\|_{L^{\infty}(\partial_p Q_1^+\cap \{x_n=0\})} \to 0$ as $j \to \infty$, the stability theorem yields that 
\begin{equation*}
	\left\{\begin{aligned}
		\overline{u} & \in S(\lambda,\Lambda,0) && \text{in } Q_{3/4}^+ \\
		\overline{u} &= 0 && \text{on } \partial_p Q_{3/4}^+ \cap \{x_n=0\},
	\end{aligned}\right.
\end{equation*}

Then, by \Cref{lem:c1a_f0_g0}, there exist $\overline{\alpha}\in(0,1)$ and $C_1>1$ such that
\begin{align*}
	|\overline{u} (x,t)- \overline{a} x_n | \leq C_1(|x|+|t|^{\frac{1}{2-\gamma}})^{1+\overline{\alpha}} \quad \text{for all $(x,t) \in \overline{Q_{1/2}^+}$}
\end{align*}
and $|\overline{a}| \leq C_1$. For any $\alpha \in (0,\overline{\alpha})$, we take $\eta$ small enough so that $2^{1+\overline{\alpha}} \eta^{\overline{\alpha}-\alpha} C_1<1/2$. Then we observe that
\begin{align*}
	\|\overline{u}-\overline{a} x_n\|_{L^{\infty}(Q_{\eta}^+)} \leq 2^{1+\overline{\alpha}} \eta^{1+\overline{\alpha}}C_1 <\eta^{1+\alpha}/2.
\end{align*}
On the other hand, by letting $k_j \to \infty$ in \eqref{eq:contradiction}, we have
\begin{align*}
	\|\overline{u}-\overline{a} x_n\|_{L^{\infty}(Q_{\eta}^+)} \ge\eta^{1+\alpha},
\end{align*}
which leads to the contradiction.
\end{proof}
\begin{lemma} \label{lem2:gen}
	Let $\overline{\alpha}\in(0,1)$,  $C_1>1$, $\delta>0$, and $\eta >0$ be as in \Cref{lem1:gen}. Suppose that  $\alpha \in(0,\overline{\alpha})$ with $\alpha \leq 1-\gamma$ and $u \in C(\overline{Q_1^+})$ satisfies
\begin{equation*}
	\left\{\begin{aligned}
		u&\in S^*(\lambda,\Lambda, f) && \text{in } Q_1^+ \\
		u&=g && \text{on } \partial_p Q_1^+\cap \{x_n=0\}.
	\end{aligned}\right.
\end{equation*}
with $\|u\|_{L^{\infty}(Q_1^+)} \leq 1$, $\|f\|_{L^{\infty}(Q_1^+)} \leq \delta$, and
	\begin{equation*}  
		|g(x,t)| \leq \frac{\delta}{2^{1+\alpha}} (|x|+|t|^{\frac{1}{2-\gamma}})^{1+\alpha} \quad \text{for all } (x,t) \in \partial_p Q_1^+\cap \{x_n=0\}.
	\end{equation*}
Then there exists a sequence $\{a_k\}_{k=-1}^{\infty}$
such that for all $k \geq 0$, we have 
\begin{align}\label{eq:lin1_bd}
	\|u-a_kx_n\|_{L^{\infty}(Q_{\eta^k}^+)} \leq \eta^{k(1+\alpha)} \quad \text{and} \quad  |a_k-a_{k-1}| \leq  C_1 \eta^{(k-1) \alpha}.
\end{align}
\end{lemma}
 
\begin{proof}
	We argue by induction. For $k=0$, by setting $a_{-1}=a_0=0$, the conditions immediately hold. Suppose that the conclusion holds for $k \geq 0$. We claim that the conclusion also holds for $(k+1)$.
	
	For this purpose, let $r=\eta^k$, $y=x/r$, $s=t/r^{2-\gamma}$, and 
	\begin{align*}
		v(y,s) \coloneqq \frac{u(x,t)-a_kx_n}{r^{1+\alpha}}.
	\end{align*}
Then $v$ satisfies
	\begin{equation*}
		\left\{\begin{aligned}
			v & \in S^*(\lambda,\Lambda, \widetilde{f}) && \text{in } Q_1^+ \\
			v &= \widetilde{g} && \text{on } \partial_p Q_1^+\cap \{y_n=0\},
		\end{aligned}\right.
	\end{equation*}
where
\begin{equation*}
	\widetilde{f}(y,s)\coloneqq r^{1-\alpha-\gamma}f(x,t) \quad \text{and} \quad  \widetilde{g}(y,s)\coloneqq \frac{g(x,t)-a_kx_n}{r^{1+\alpha}}.
\end{equation*}
Then it immediately follows that $\|v\|_{L^{\infty}(Q_1^+)} \leq 1$, $\|\widetilde{f}\|_{L^{\infty}(Q_1^+)} \leq \delta$, and 
\begin{align*}
	 |\widetilde{g}(y,s)| & \leq \frac{1}{r^{1+\alpha}}\cdot \frac{\delta}{2^{1+\alpha}} \cdot 2^{1+\alpha}r^{1+\alpha}  = \delta \quad \text{for all } (y,s) \in \partial_p Q_1^+\cap \{y_n=0\}.
\end{align*}
Thus, by applying \Cref{lem1:gen} to $v$, there exists a constant $\widetilde{a}\in\mathbb{R}$ such that 
\begin{align}\label{eq:induc1_bd}
	\|v-\widetilde{a} y_n\|_{L^{\infty}(Q_{\eta}^+)} \leq \eta^{1+\alpha}
\end{align}
and
\begin{align}\label{eq:induc2_bd}
	|\widetilde{a}| \leq C_1.
\end{align}
We now let $a_{k+1} \coloneqq a_k +r^{\alpha} \widetilde{a}$. Then \eqref{eq:induc2_bd} implies that
\begin{align*}
	|a_{k+1}-a_k|  \leq  C_1 \eta^{k\alpha}.
\end{align*}
Finally, \eqref{eq:induc1_bd} shows that  
\begin{align*}
	|u(x,t)-a_{k+1}x_n| &=|u(x,t)-a_kx_n-r^{\alpha}\widetilde{a}x_n | \\
	&\leq r^{1+\alpha} |v(x/r,t/r^{2-\gamma})-\widetilde{a}(x_n/r)| \\
	&\leq \eta^{(k+1)(1+\alpha)}
\end{align*}
for all $(x,t) \in  Q_{\eta^{k+1}}^+$, as desired.
\end{proof}
We are now ready to prove the first main theorem, \Cref{thm:main}.
\begin{proof}[Proof of \Cref{thm:main}]
	Without loss of generality, we can assume that $(x_0,t_0)=(0,0)$. We claim that the assumptions in \Cref{lem2:gen} hold after the appropriate reduction argument. Since $g\in C^{1,\alpha}(0,0)$, there exists a linear function $L_g$ such that
	\begin{equation*}
		|g(x,t)-L_g(x)| \leq [g]_{C^{1,\alpha}(0,0)} (|x|+|t|^{\frac{1}{2-\gamma}})^{1+\alpha} \quad \text{for all }(x,t) \in \partial_p Q_1^+ \cap \{x_n=0\}.
	\end{equation*}
	Then $\overline{u} \coloneqq u-L_g$ satisfies 
	\begin{equation*} 
		\left\{\begin{aligned}
			\overline{u} & \in S^*(\lambda,\Lambda, f) && \text{in } Q_1^+  \\
			\overline{u} &= \overline{g} && \text{on } \partial_p Q_1^+ \cap \{x_n=0\},
		\end{aligned}\right.
	\end{equation*}
	where $\overline{g} \coloneqq g - L_g$ satisfies
	\begin{equation*}
		|\overline{g}(x,t)| \leq [g]_{C^{1,\alpha}(0,0)} (|x|+|t|^{\frac{1}{2-\gamma}})^{1+\alpha} \quad \text{for all }(x,t) \in \partial_p Q_1^+ \cap \{x_n=0\} .
	\end{equation*}
	For $K\coloneqq \|\overline{u}\|_{L^{\infty}(Q_1^+)}+\|f\|_{L^{\infty}(Q_1^+)}/\delta + 2^{1+\alpha}\|\overline{g}\|_{C^{1,\alpha}(0,0)}/\delta+1$, by considering  
	\begin{align*}
		\widetilde{u}(x,t) \coloneqq \overline{u}(x,t)/K, \quad \widetilde{f}(x,t) \coloneqq f(x,t)/K, \quad \text{and} \quad \widetilde{g}(x,t) \coloneqq \overline{g}(x,t)/K,
	\end{align*}
we may assume that $\|u\|_{L^{\infty}(Q_1^+)} \leq 1$, $\|f\|_{L^{\infty}(Q_1^+)} \leq \delta$, and 
	\begin{equation*} 
		|g(x,t)| \leq \frac{\delta}{2^{1+\alpha}} (|x|+|t|^{\frac{1}{2-\gamma}})^{1+\alpha} \quad \text{for all }  (x,t) \in  \partial_p Q_1^+ \cap \{x_n=0\}.
	\end{equation*}
Thus, we can apply \Cref{lem2:gen} for $u$ to find a sequence $\{a_k\}_{k=-1}^{\infty}$ satisfying \eqref{eq:lin1_bd}. Indeed, by standard argument, we obtain limit $a$ such that $a_k \to a$ satisfying
\begin{align*}
	|a_k-a| \leq C\eta^{k \alpha }.
\end{align*}
Finally, for any $(x,t) \in Q_1^+$, there exists $j \geq 0$ such that $\eta^{j+1} \leq \max\{ |x|, |t|^{\frac{1}{2-\gamma}} \}< \eta^j$. Then we conclude that
\begin{align*}
	|u(x,t)-ax_n| \leq |u(x,t)-a_j x_n|+|a_j -a| |x_n| \leq C(|x|+|t|^{\frac{1}{2-\gamma}})^{1+\alpha},
\end{align*}
which implies that $u$ is $C^{1, \alpha}$ at $(0,0)$.
\end{proof}
%
%
\section{The Cauchy-Dirichlet Problem} \label{sec:dirichlet}

In this section, we prove the boundary $C^{2,\alpha}$-regularity for solutions of \eqref{eq:c2a_sg} and the solvability of the Cauchy-Dirichlet problems \eqref{eq:diri_model}. Unlike boundary $C^{1,\alpha}$-regularity, when studying boundary $C^{2,\alpha}$-regularity, the conditions for $f$ in the degenerate equation and the singular equation are different. We first prove two lemmas that can be applied in common, and it is proved by dividing the case of degenerate equations and the case of singular equations.

Note that for any $\Omega\ssubset \mathbb{R}_+^n$, \eqref{eq:model} is a uniformly parabolic equation in $\Omega_T$, so we know that solutions $u$ of \eqref{eq:model} are in $C^2(Q_1^+)$.
\begin{lemma} \label{lem:c2a_fz}
Let $u \in C^2(Q_1^+)\cap C(\overline{Q_1^+})$ be a viscosity solution of  
\begin{equation} \label{eq:model_fz}
	\left\{\begin{aligned}
		u_t- x_n^{\gamma} F(D^2 u)&= 0 && \text{in } Q_1^+ \\
		u&=0 && \text{on } \partial_p Q_1^+\cap \{x_n=0\}.
	\end{aligned}\right.
\end{equation}
Then $u \in C^{2,\alpha}(0,0)$ for any $\alpha \in(0,\overline{\alpha})$, i.e., there exists a polynomial $P(x)$ with $\deg P \leq 2$ such that
\begin{equation*}
	|u(x,t)-P(x)| \leq C_2( \|u\|_{L^{\infty}(Q_1^+)} + |F(O_n)|) |x|^{1+\alpha} x_n \quad \text{for all } (x,t) \in \overline{Q_{1/2}^+},
\end{equation*}
$\|P\| \leq C_2$ and $F(D^2 P) =0$, where $C_2>1$ is a constant depending only on $n$, $\lambda$, $\Lambda$, $\gamma$, and $\alpha$. 

Moreover, $u_t \in C^{1,\overline{\alpha}}(0,0)$, i.e., there exists a constant $a \in \mathbb{R}$ such that
\begin{equation*}
	|u_t(x,t)-ax_n| \leq C_3( \|u\|_{L^{\infty}(Q_1^+)} + |F(O_n)|) (|x| + |t|^{\frac{1}{2-\gamma}})^{\overline{\alpha}} x_n \quad \text{for all } (x,t) \in \overline{Q_{1/2}^+}
\end{equation*}
and $|a| \leq C_3$, where $C_3>1$ is a constant depending only on $n$, $\lambda$, $\Lambda$, $\gamma$, and $\overline{\alpha}$.
\end{lemma}

\begin{proof}
	Without loss of generality, we may assume that $\|u\|_{L^{\infty}(Q_1^+)}+|F(O_n)| \leq 1$. Consider the following difference quotient 
	\begin{equation*}
		v^h(x,t) \coloneqq \frac{u(x,t+h)- u(x,t)}{h} \quad \text{for } h \in (0,1/4).
	\end{equation*}
	Since $v^h \in C^2(Q_{3/4}^+)$, the following inequalities hold in the classical sense:
	\begin{equation*}
		x_n^{\gamma} \mathcal{M}^-_{\lambda,\Lambda}  (D^2 v^h ) \leq v_t^h \leq x_n^{\gamma} \mathcal{M}^+_{\lambda,\Lambda}  (D^2 v^h ) \quad \text{in } Q_{3/4}^+.
	\end{equation*}
Thus, the stability theorem yields that 
\begin{equation} \label{eq:ut}
	\left\{\begin{aligned}
		u_t &\in S(\lambda,\Lambda, 0) && \text{in } Q_{3/4}^+ \\
		u_t &=0 && \text{on } \partial_p Q_{3/4}^+\cap \{x_n=0\}.
	\end{aligned}\right.
\end{equation}
By Lipschitz estimate [\Cref{lem:u<x}], we have
\begin{equation*}
	|x_n^{-\gamma}u_t(x,t)| \leq C_1 x_n^{1-\gamma} \leq C_1 \quad \text{for all } (x,t) \in \overline{Q_{1/2}^+},
\end{equation*}
where $C_1>1$ is a constant depending only on $n$, $\lambda$, $\Lambda$, and $\gamma$. So, the parabolic Dirichlet problem \eqref{eq:model_fz} can be regarded as the following elliptic Dirichlet problem: 
\begin{equation} \label{eq:model_ell}
	\left\{\begin{aligned}
		F(D^2 u_0)&= f && \text{in } B_{1/2}^+ \\
		u_0 &=0 && \text{on } \partial B_{1/2}^+\cap \{x_n=0\},
	\end{aligned}\right.
\end{equation}
where $u_0(x) = u(x,0)$ and $f(x) = x_n^{-\gamma}u_t(x,0)$. By the result of the boundary regularity of the elliptic equations; see \cite[Theorem 1.2]{SS14} or \cite[Theorem 1.8]{LZ20}, we know that for any $\alpha \in(0,\overline{\alpha})$, there exists a polynomial $P(x)$ with $\deg P \leq 2$ such that
\begin{equation*}
	|u(x,0)-P(x)| \leq C_2 |x|^{2+\alpha} \quad \text{for all } (x,t) \in \overline{B_{1/2}^+}
\end{equation*}
and $\|P\| \leq C_2$, where $C_2>1$ is a constant depending only on $n$, $\lambda$, $\Lambda$, $\gamma$, and $\alpha$. For the same reason as \Cref{rmk_trans_inv}, we know that 
\begin{equation*}
	|u(x,t)-P(x)| \leq C_2 |x|^{1+\alpha} x_n \quad \text{for all } (x,t) \in \overline{Q_{1/2}^+}
\end{equation*}
holds by considering the translation in the $x'$ and $t$ directions. Moreover, $F(D^2 P) =0$ holds by the definition of viscosity solution.

Finally, since $u_t$ satisfies \eqref{eq:ut}, we can apply \Cref{lem:c1a_f0_g0} to $u_t$, and hence there exists a constant $a \in \mathbb{R}$ such that
\begin{equation*}
	|u_t(x,t)-ax_n| \leq C_3 (|x| + |t|^{\frac{1}{2-\gamma}})^{\overline{\alpha}} x_n \quad \text{for all } (x,t) \in \overline{Q_{1/2}^+}
\end{equation*}
and $|a| \leq C_3$, where $C_3>1$ is a constant depending only on $n$, $\lambda$, $\Lambda$, $\gamma$, and $\overline{\alpha}$.
\end{proof}

\begin{lemma} \label{lem1:gen_c2a}
	Let $C_2>1$ be as in \Cref{lem:c2a_fz}. For any $\alpha \in (0,\overline{\alpha})$ with $\alpha\leq 1-\gamma$, there exists $\delta >0$ depending only on $n$, $\lambda$, $\Lambda$, $\gamma$, and $\alpha$ such that if $u \in C^2(Q_1^+)\cap C(\overline{Q_1^+})$ is a viscosity solution of 
\begin{equation*}
	\left\{\begin{aligned}
		u_t - x_n^{\gamma}F(D^2 u,x,t)& = f && \text{in } Q_1^+ \\
		u&=g && \text{on } \partial_p Q_1^+\cap \{x_n=0\}
	\end{aligned}\right.
\end{equation*}
with $\|u\|_{L^{\infty}(Q_1^+)} + |F(O_n,0,0)| \leq 1$, $u(0,0) =0$, $Du(0,0)=0$, $\|\beta^1\|_{L^{\infty}(Q_1^+)}\leq \delta$, $\|\beta^2\|_{L^{\infty}(Q_1^+)}\leq \delta$, $\|f\|_{L^{\infty}(Q_1^+)}\leq\delta$, and $\|g\|_{L^{\infty}(\partial_p Q_1^+\cap \{x_n=0\})} + \|g\|_{C^{1,\alpha}(0,0)} \leq \delta$, then there exists a polynomial $P(x)= \sum_{i=1}^n a_i x_i x_n$ such that
\begin{equation*}
	\|u- P \|_{L^{\infty}(Q_{\eta}^+)} \leq \eta^{2+\alpha},
\end{equation*}
$\sum_{i=1}^n |a_i| \leq C_2$, and $F(D^2 P,0,0)=0$, where $\eta>0$ is a constant depending only on $n$, $\lambda$, $\Lambda$, $\gamma$, and $\alpha$.
\end{lemma}

\begin{proof}
We prove by contradiction; we suppose the conclusion does not hold. In other words, there exist sequences $u_k$, $f_k$, $g_k$, $\beta^1_k$, $\beta^2_k$, and $F_k$ such that $u_k$ is a viscosity solution of 
	\begin{equation*}
		\left\{\begin{aligned}
			\partial_ t u_k - x_n^{\gamma} F_k(D^2 u_k,x,t) &= f_k && \text{in } Q_1^+ \\
			u_k &= g_k && \text{on } \partial_p Q_1^+ \cap \{x_n=0\},
		\end{aligned}\right.
	\end{equation*}
with $\|u_k\|_{L^{\infty}(Q_1^+)} + |F_k(O_n,0,0)| \leq 1$, $u_k(0,0) =0$, $Du_k(0,0)=0$, $\|\beta^1_k\|_{L^{\infty}(Q_1^+)}\leq 1/k$, $\|\beta^2_k\|_{L^{\infty}(Q_1^+)}\leq 1/k$, $\|f_k\|_{L^{\infty}(Q_1^+)}\leq 1/k$, and $\|g_k\|_{L^{\infty}(\partial_p Q_1^+\cap \{x_n=0\})} + \|g_k\|_{C^{1,\alpha}(0,0)} \leq 1/k$. Moreover, for any polynomial $P(x)= \sum_{i=1}^n a_i x_i x_n$ satisfying $\sum_{i=1}^n|a_i| \leq C_2$ and $F_k(D^2 P,0,0)=0$, we have
\begin{align}\label{eq:cont_c2a}
	\|u_k- P\|_{L^{\infty}(Q_{\eta}^+)} >\eta^{2+\alpha},
\end{align}
where $\eta \in (0,1)$ will be determined later.

As in the proof of \Cref{lem1:gen}, $\{u_k\}_{k=1}^{\infty}$ is uniformly bounded in $L^{\infty}(Q_1^+)$ and equicontinuous in compact sets of $Q_1^+$. Thus, by Arzela-Ascoli theorem, there exist a subsequence $\{u_{k_j}\}_{j=1}^{\infty}$ of $\{u_k\}_{k=1}^{\infty}$ and a limit function $\overline{u} \in C^2(Q_1^+) \cap C(\overline{Q_1^+})$ such that $u_{k_j} \to \overline{u}$ uniformly in compact sets of $Q_1^+$ as $j \to \infty$. 

Moreover, since $|F_k(O_n,0,0)| \leq 1$ and $F_k$ are Lipschitz continuous in $M$ with a uniform Lipschitz constant, $\{F_k\}_{k=1}^{\infty}$ is also uniformly bounded and equicontinuous in compact sets of $\mathcal{S}^n$. Then there exist a subsequence $\{F_{k_j}\}_{j=1}^{\infty}$ of $\{F_k\}_{k=1}^{\infty}$ and a limit operator $\overline{F}$ such that $F_{k_j}(\cdot,0,0) \to \overline{F}$ uniformly in compact sets of $\mathcal{S}^n$ as $j \to \infty$. It follows that 
\begin{align*}
	|F_{k_j}(M,x,t) - \overline{F}(M)| &\leq |F_{k_j}(M,x,t) - F_{k_j}(M,0,0)| + |F_{k_j}(M,0,0)- \overline{F}(M)| \\
	&\leq \beta^1_{k_j}(x,t) \|M\| + \beta^2_{k_j}(x,t) + |F_{k_j}(M,0,0)- \overline{F}(M)| \\
	&\leq \frac{1}{k_j} (\|M\| +1)+ |F_{k_j}(M,0,0)- \overline{F}(M)| 
\end{align*}
for all $M\in\mathcal{S}^n$, $(x,t) \in \overline{Q_1^+}$. Since the right-hand side above tends to zero uniformly in compact sets in $\mathcal{S}^n$ as $j \to \infty$, the stability theorem yields that 
\begin{equation*}
	\left\{\begin{aligned}
		\overline{u}_t -x_n^{\gamma} \overline{F}(D^2 \overline{u})&=0  && \text{in } Q_{3/4}^+ \\
		\overline{u} &= 0 && \text{on } \partial_p Q_{3/4}^+ \cap \{x_n=0\},
	\end{aligned}\right.
\end{equation*}
Since $u_{k_j}(0,0) =0$ and $Du_{k_j}(0,0)=0$, by applying \Cref{thm:main} to $u_{k_j}$, we have
\begin{equation} \label{kth_r1a}
	\|u_{k_j}\|_{L^{\infty}(Q_r^+)} \leq C r^{1+\alpha} \quad \text{for all } r \in (0,1),
\end{equation}
where $C>0$ is a constant depending only on $n$, $\lambda$, $\Lambda$, $\gamma$, and $\alpha$. Since $u_{k_j} \to \overline{u}$ uniformly, by letting $j \to \infty$ in \eqref{kth_r1a}, we obtain
\begin{equation*}
	\|\overline{u}\|_{L^{\infty}(Q_r^+)} \leq C r^{1+\alpha} \quad \text{for all } r \in (0,3/4).
\end{equation*}
Thus, we have $\overline{u}(0,0)=0$ and $D\overline{u}(0,0)=0$. Then, by \Cref{lem:c2a_fz}, there exists a polynomial $\overline{P}(x)=\sum_{i=1}^n \overline{a}_i x_i x_n$ such that
\begin{equation*}
	|\overline{u}(x,t)-\overline{P}(x)| \leq C_2 |x|^{1+\frac{\alpha+\overline{\alpha}}{2}} x_n \quad \text{for all } (x,t) \in \overline{Q_{1/2}^+}
\end{equation*}
and $\sum_{i=1}^n|\overline{a}_i| \leq C_2$ and $\overline{F}(D^2\overline{P})=0$. 

Now, we take $\eta$ small enough so that $n^{\frac{1}{2} + \frac{\alpha + \overline{\alpha}}{4}} \eta^{\frac{\overline{\alpha}-\alpha}{2}} C_2<1/2$. Then we observe that
\begin{align*}
	\|\overline{u}-\overline{P}\|_{L^{\infty}(Q_{\eta}^+)} <\eta^{2+\alpha}/2.
\end{align*}
On the other hand, since $F_{k_j}(D^2 \overline{P},0,0) \to \overline{F}(D^2\overline{P}) = 0$ as $j \to \infty$, there exists $\theta_{k_j} $ with $\theta_{k_j} \to 0$ and $|\theta_{k_j}| \leq 1$ such that $F_{k_j}(D^2\overline{P}+\theta_{k_j} I_n,0,0) = 0$. From \eqref{eq:cont_c2a}, we have
 \begin{equation*}
	\|u_{k_j}- \overline{P} - \theta_{k_j} |x|^2/2\|_{L^{\infty}(Q_{\eta}^+)} >\eta^{2+\alpha},
\end{equation*}
and hence by letting $k_j \to \infty$, we obtain
\begin{align*}
	\|\overline{u}-\overline{P}\|_{L^{\infty}(Q_{\eta}^+)} \ge\eta^{2+\alpha},
\end{align*}
which leads to the contradiction.
\end{proof}
\subsection{Degenerate equations}
In this subsection, we prove the boundary $C^{2,\alpha}$-regularity for solutions of degenerate equations. That is, it is assumed that $0<\gamma<1$. As a reminder, according to \Cref{rmk1}, the regularity of solutions of \eqref{eq:model} is at most $C^{1,1-\gamma}$, when the decay order of $x_n$ for $f$ is strictly less than $\gamma$. Thus, in dealing with degenerate equations, we assume that $f \equiv 0$. Actually, since $x_n^{\gamma} F(O_n,x,t)$ plays the role of $f$, it is the same as considering $f$ whose decay order of $x_n$ is $\gamma$.

\begin{lemma} \label{lem2:gen_c2a}
	Let  $C_2>1$, $\delta>0$, and $\eta >0$ be as in \Cref{lem1:gen_c2a}. Suppose that  $\alpha \in(0,\overline{\alpha})$ with $\alpha \leq 1-\gamma$ and $u \in  C^2(Q_1^+)\cap C(\overline{Q_1^+})$ is a viscosity solution of
\begin{equation*}
	\left\{\begin{aligned}
		u_t - x_n^{\gamma} F(D^2u,x,t)&=0 && \text{in } Q_1^+ \\
		u&=g && \text{on } \partial_p Q_1^+\cap \{x_n=0\}.
	\end{aligned}\right.
\end{equation*} 
with $\|u\|_{L^{\infty}(Q_1^+)} + |F(O_n,0,0)| \leq 1$, $u(0,0) =0$, $Du(0,0)=0$,
	\begin{align} 
		|\beta^1(x,t)| &\leq \frac{\delta}{K2^{\alpha+1}} (|x|+|t|^{\frac{1}{2-\gamma}})^{\alpha} \quad \text{for all } (x,t) \in \overline{Q_1^+} , \label{beta_hol1} \\
		|\beta^2(x,t)| &\leq \frac{\delta}{2^{\alpha+1}} (|x|+|t|^{\frac{1}{2-\gamma}})^{\alpha} \quad \text{for all } (x,t) \in \overline{Q_1^+},  \label{beta_hol2}
	\end{align}  
and
	\begin{align*}  
		|g(x,t)| &\leq \frac{\delta}{2^{2+\alpha}} (|x|+|t|^{\frac{1}{2-\gamma}})^{2+\alpha} \quad \text{for all } (x,t) \in \partial_p Q_1^+\cap \{x_n=0\},
	\end{align*}
	where $K>1$ is a constant depending only on $n$, $\lambda$, $\Lambda$, $\gamma$, and $\alpha$. Then there exists a sequence $\{P_k\}_{k=-1}^{\infty}$ of the form $P_k(x)= \sum_{i=1}^n a_i^k x_i x_n$
such that for all $k \geq 0$, we have 
\begin{align}\label{eq:quad1_bd}
	\|u-P_k\|_{L^{\infty}(Q_{\eta^k}^+)} \leq \eta^{k(2+\alpha)}, \quad  \sum_{i=1}^n |a^k_i-a^{k-1}_i| \leq  C_2 \eta^{(k-1)\alpha},  \quad \text{and} \quad F(D^2P_k,0,0)=0.
\end{align}
\end{lemma}

\begin{proof}
	We argue by induction. For $k=0$, by setting $P_{-1}=P_0\equiv0$, the conditions immediately hold. Suppose that the conclusion holds for $k \geq 0$. We claim that the conclusion also holds for $(k+1)$.
	
	For this purpose, let $r=\eta^k$, $y=x/r$, $s=t/r^{2-\gamma}$, and 
	\begin{align*}
		v(y,s) \coloneqq \frac{u(x,t)-P_k(x)}{r^{2+\alpha}}.
	\end{align*}
Then $v \in  C^2(Q_1^+)\cap C(\overline{Q_1^+})$ is a viscosity solution of
	\begin{equation*}
		\left\{\begin{aligned}
			v_s - y_n^{\gamma} \widetilde{F}(D^2 v,y,s)& = 0 && \text{in } Q_1^+ \\
			v &= \widetilde{g} && \text{on } \partial_p Q_1^+\cap \{y_n=0\},
		\end{aligned}\right.
	\end{equation*}
where
\begin{equation*}
	\widetilde{F}(M,y,s)\coloneqq r^{-\alpha}F(r^\alpha M + D^2 P_k,x,t) \quad \text{and} \quad  \widetilde{g}(y,s)\coloneqq \frac{g(x,t)-P_k(x)}{r^{2+\alpha}}.
\end{equation*}
Then it immediately follows that $\|v\|_{L^{\infty}(Q_1^+)}+|\widetilde{F}(O_n,0,0)| \leq 1$, $v(0,0)=0$, $Dv(0,0)=0$, and 
\begin{align*}
	 |\widetilde{g}(y,s)| & \leq \frac{1}{r^{2+\alpha}}\cdot \frac{\delta}{2^{2+\alpha}} \cdot 2^{2+\alpha}r^{2+\alpha}  = \delta \quad \text{for all } (y,s) \in \partial_p Q_1^+\cap \{y_n=0\}.
\end{align*}
By the second inequality in \eqref{eq:quad1_bd}, there exists a constant $K>1$ depending only on $n$, $\lambda$, $\Lambda$, $\gamma$, and $\alpha$ such that $\|D^2 P_k\| \leq K$. Define 
\begin{equation*}
	\widetilde{\beta}^1(y,s) \coloneqq \beta^1(x,t)  \quad \text{and} \quad \widetilde{\beta}^2(y,s) \coloneqq  r^{-\alpha} K \beta^1(x,t) +r^{-\alpha} \beta^2(x,t).
\end{equation*}
Then it follows that $\|\widetilde{\beta}^1\|_{L^{\infty}(Q_1^+)}\leq \delta$ and $\|\widetilde{\beta}^2\|_{L^{\infty}(Q_1^+)}\leq \delta$. From \eqref{beta_hol1} and \eqref{beta_hol2}, we have
\begin{align*}
	|\widetilde{F}(M,y,s)- \widetilde{F}(M,0,0)|&\leq r^{-\alpha} |F(r^\alpha M + D^2 P_k,x,t)- F(r^\alpha M + D^2 P_k,0,0)| \\
	&\leq r^{-\alpha} \beta^1(x,t) \|r^\alpha M + D^2 P_k\| +  r^{-\alpha} \beta^2(x,t) \\
	&\leq\widetilde{\beta}^1(y,s) \|M\| + \widetilde{\beta}^2(y,s)
\end{align*}
Thus, by applying \Cref{lem1:gen_c2a} to $v$, there exists a polynomial $\widetilde{P}(y)=\sum_{i=1}^n \widetilde{a}_i y_i y_n$ such that 
\begin{align}\label{eq:ind1_bd}
	\|v-\widetilde{P}\|_{L^{\infty}(Q_{\eta}^+)} \leq \eta^{2+\alpha},
\end{align}
\begin{align}\label{eq:ind2_bd}
	\sum_{i=1}^n| \widetilde{a}_i| \leq C_2 ,\quad \text{and} \quad \widetilde{F}(D^2 \widetilde{P},0,0)=0.
\end{align}
We now let $P_{k+1}(x) \coloneqq P_k(x) +r^{2+\alpha} \widetilde{P}(x/r)$. Then \eqref{eq:ind2_bd} implies that
\begin{align*}
	\sum_{i=1}^n |a^{k+1}_i-a^k_i|  \leq  C_2 \eta^{k\alpha} \quad \text{and} \quad F(D^2P_{k+1},0,0) = 0.
\end{align*}
Finally, \eqref{eq:ind1_bd} shows that  
\begin{align*}
	|u(x,t)-P_{k+1}(x)| &=|u(x,t)- P_k(x)-r^{2+\alpha}\widetilde{P}(x/r) | \\
	&\leq r^{2+\alpha} |v(x/r,t/r^{2-\gamma})-\widetilde{P}(x/r)| \\
	&\leq \eta^{(k+1)(2+\alpha)}
\end{align*}
for all $(x,t) \in  Q_{\eta^{k+1}}^+$, as desired.
\end{proof}

We are now ready to prove the second main theorem, \Cref{thm:main2}.
\begin{proof}[Proof of \Cref{thm:main2}]
	Without loss of generality, we can assume that $(x_0,t_0)=(0,0)$. We claim that the assumptions in \Cref{lem2:gen_c2a}  hold after the appropriate reduction argument. Since $g\in C^{2,\alpha}(0,0)$ and $g_t(0,0)=0$, there exists a polynomial $P_g$ with $\det P_g \leq 2$ such that
	\begin{equation*}
		|g(x,t)-P_g(x)| \leq [g]_{C^{2,\alpha}(0,0)} (|x|+|t|^{\frac{1}{2-\gamma}})^{2+\alpha} \quad \text{for all }(x,t) \in \partial_p Q_1^+ \cap \{x_n=0\}.
	\end{equation*}
	
	By \Cref{thm:main}, $Du(0,0)$ is well-defined and the function $\overline{u}(x,t) \coloneqq u(x,t)-P_g(x) - u_n(0,0)x_n$ is a viscosity solution of  
	\begin{equation*} 
		\left\{\begin{aligned}
			\overline{u}_t -x_n^{\gamma} \overline{F}(D^2 \overline{u},x,t ) & =0 && \text{in } Q_1^+  \\
			\overline{u} &= \overline{g} && \text{on } \partial_p Q_1^+ \cap \{x_n=0\},
		\end{aligned}\right.
	\end{equation*}
	where $\overline{F}(M,x,t)\coloneqq F(M + D^2 P_g,x,t)$ and $\overline{g}(x,t) \coloneqq g(x,t) - P_g(x) - u_n(0,0)x_n$ which satisfies
	\begin{equation*}
		|\overline{g}(x,t)| \leq [g]_{C^{2,\alpha}(0,0)} (|x|+|t|^{\frac{1}{2-\gamma}})^{2+\alpha} \quad \text{for all }(x,t) \in \partial_p Q_1^+ \cap \{x_n=0\} .
	\end{equation*}
	
	Define 
\begin{equation*}
	\overline{\beta}^1(x,t) \coloneqq \beta^1(x,t)  \quad \text{and} \quad \overline{\beta}^2(x,t) \coloneqq  \beta^1(x,t) \|D^2 P_g\|+ \beta^2(x,t).
\end{equation*}
Then it follows that $|\overline{F}(M,x,t)- \overline{F}(M,0,0)| \leq \overline{\beta}^1(x,t) \|M\| + \overline{\beta}^2(x,t)$.
Let
	\begin{align*}
		K&\coloneqq \|\overline{u}\|_{L^{\infty}(Q_1^+)} + \delta^{-1}2^{\alpha+1}  2[\overline{\beta^2}]_{C^{\alpha}(0,0)} +\delta^{-1}2^{2+\alpha} \|\overline{g}\|_{C^{2,\alpha}(0,0)}  + \overline{F}(O_n,0,0),
	\end{align*}
by considering 
	\begin{align*}
		\widetilde{u}(x,t) \coloneqq \overline{u}(x,t)/K  \quad \text{and} \quad \widetilde{g}(x,t) \coloneqq \overline{g}(x,t)/K,
	\end{align*}
	and the scaling argument, we may assume that $\|u\|_{L^{\infty}(Q_1^+)} + |F(O_n,0,0)|\leq 1$, $u(0,0)=0$, $Du(0,0)=0$,
	\begin{align*} 
		|\beta^1(x,t)| &\leq \frac{\delta}{K2^{\alpha+1}} (|x|+|t|^{\frac{1}{2-\gamma}})^{\alpha} \quad \text{for all } (x,t) \in \overline{Q_1^+} ,   \\
		|\beta^2(x,t)| &\leq \frac{\delta}{2^{\alpha+1}} (|x|+|t|^{\frac{1}{2-\gamma}})^{\alpha} \quad \text{for all } (x,t) \in \overline{Q_1^+},  	
	\end{align*}  
and 
	\begin{equation*} 
		|g(x,t)| \leq \frac{\delta}{2^{2+\alpha}} (|x|+|t|^{\frac{1}{2-\gamma}})^{2+\alpha} \quad \text{for all }  (x,t) \in  \partial_p Q_1^+ \cap \{x_n=0\}.
	\end{equation*}
	
Thus, we can apply \Cref{lem2:gen_c2a} for $u$ to find a sequence $\{P_k\}_{k=-1}^{\infty}$ of the form $P_k(x)= \sum_{i=1}^n a_i^k x_i x_n$
such that for all $k \geq 0$, we have 
\begin{align*} 
	\|u-P_k\|_{L^{\infty}(Q_{\eta^k}^+)} \leq \eta^{k(2+\alpha)}, \quad  \sum_{i=1}^n |a^k_i-a^{k-1}_i| \leq  C_2 \eta^{(k-1)\alpha},  \quad \text{and} \quad F(D^2P_k,0,0)=0.
\end{align*}
Indeed, by standard argument, we obtain a polynomial $P(x)=\sum_{i=1}^n a_i x_i x_n$ such that $P_k \to P$ satisfying
\begin{align*}
	\sum_{i=1}^n |a^k_i-a_i| \leq C_2\eta^{k \alpha } \quad \text{and} \quad  F(D^2 P,0,0)=0.
\end{align*}
Finally, for any $(x,t) \in Q_1^+$, there exists $j \geq 0$ such that $\eta^{j+1} \leq \max\{ |x|, |t|^{\frac{1}{2-\gamma}} \}< \eta^j$. Then we conclude that
\begin{align*}
	|u(x,t)-P(x)| \leq |u(x,t)-P_j(x)|+|P_j(x) -P(x)|  \leq C(|x|+|t|^{\frac{1}{2-\gamma}})^{2+\alpha},
\end{align*}
which implies that $u$ is $C^{2, \alpha}$ at $(0,0)$.
\end{proof}

\subsection{Singular equations}
In this subsection, we prove the boundary $C^{2,\alpha}$-regularity for solutions of singular equations. That is, it is assumed that $\gamma<0$. The proof is similar to that of degenerate equations. Thus, we give a sketch of proofs and focus on applying the condition on $f$ in the proof.
\begin{theorem} \label{lem2:gen_c2a_sg}
	Let  $C_2>1$, $\delta>0$, and $\eta >0$ be as in \Cref{lem1:gen_c2a}. Suppose that  $\alpha \in(0,\overline{\alpha})$ and $u \in  C^2(Q_1^+)\cap C(\overline{Q_1^+})$ is a viscosity solution of
\begin{equation*}
	\left\{\begin{aligned}
		u_t - x_n^{\gamma} F(D^2u,x,t)&=f && \text{in } Q_1^+ \\
		u&=g && \text{on } \partial_p Q_1^+\cap \{x_n=0\}.
	\end{aligned}\right.
\end{equation*} 
with $\|u\|_{L^{\infty}(Q_1^+)} + |F(O_n,0,0)| \leq 1$, $u(0,0) =0$, $Du(0,0)=0$,
	\begin{align} 
		|\beta^1(x,t)| &\leq \frac{\delta}{K2^{\alpha+1}} (|x|+|t|^{\frac{1}{2-\gamma}})^{\alpha} \quad \text{for all } (x,t) \in \overline{Q_1^+} ,\nonumber\\
		|\beta^2(x,t)| &\leq \frac{\delta}{2^{\alpha+1}} (|x|+|t|^{\frac{1}{2-\gamma}})^{\alpha} \quad \text{for all } (x,t) \in \overline{Q_1^+}, \nonumber \\
		|f(x,t)| &\leq \frac{\delta}{2^{\alpha+\gamma}} (|x|+|t|^{\frac{1}{2-\gamma}})^{\alpha+\gamma} \quad \text{for all } (x,t) \in \overline{Q_1^+},  \label{f_hol_sg}
	\end{align}  
and
	\begin{align*}  
		|g(x,t)| &\leq \frac{\delta}{2^{2+\alpha}} (|x|+|t|^{\frac{1}{2-\gamma}})^{2+\alpha} \quad \text{for all } (x,t) \in \partial_p Q_1^+\cap \{x_n=0\},
	\end{align*}
	where $K>1$ is a constant depending only on $n$, $\lambda$, $\Lambda$, $\gamma$, and $\alpha$. Then there exists a sequence $\{P_k\}_{k=-1}^{\infty}$ of the form $P_k(x)= \sum_{i=1}^n a_i^k x_i x_n$
such that for all $k \geq 0$, we have 
\begin{align*}
	\|u-P_k\|_{L^{\infty}(Q_{\eta^k}^+)} \leq \eta^{k(2+\alpha)}, \quad  \sum_{i=1}^n |a^k_i-a^{k-1}_i| \leq  C_2 \eta^{(k-1)\alpha},  \quad \text{and} \quad F(D^2P_k,0,0)=0.
\end{align*}
\end{theorem}

\begin{proof}
The difference from the proof of \Cref{lem2:gen_c2a} is that the function $v \in  C^2(Q_1^+)\cap C(\overline{Q_1^+})$ is a viscosity solution of
	\begin{equation*}
		\left\{\begin{aligned}
			v_s - y_n^{\gamma} \widetilde{F}(D^2 v,y,s)& =  \widetilde{f} && \text{in } Q_1^+ \\
			v &= \widetilde{g} && \text{on } \partial_p Q_1^+\cap \{y_n=0\},
		\end{aligned}\right.
	\end{equation*}
where
\begin{equation*}
	\widetilde{F}(M,y,s)\coloneqq r^{-\alpha}F(r^\alpha M + D^2 P_k,x,t), \quad \widetilde{f}(y,s)\coloneqq \frac{f(x,t)}{r^{\alpha+\gamma}}, \quad \text{and} \quad  \widetilde{g}(y,s)\coloneqq \frac{g(x,t)-P_k(x)}{r^{2+\alpha}}.
\end{equation*}
Also, from \eqref{f_hol_sg}, the following inequality holds:
\begin{align*}
	 |\widetilde{f}(y,s)| & \leq \frac{1}{r^{\alpha+\gamma}}\cdot \frac{\delta}{2^{\alpha+\gamma}} \cdot 2^{\alpha+\gamma}r^{\alpha+\gamma}  = \delta \quad \text{for all } (y,s) \in \overline{Q_1^+} .
\end{align*}
Now if we apply \Cref{lem1:gen_c2a} to $v$, the remaining proof is exactly the same as \Cref{lem2:gen_c2a}.
\end{proof}

We are now ready to prove the third main theorem, \Cref{thm:main2_sg}.
\begin{proof}[Proof of \Cref{thm:main2_sg}]
	The difference from the proof of \Cref{thm:main2} is that the polynomial $P_g$ depends on $t$ as well as $x$ and the function $\overline{u}(x,t) \coloneqq u(x,t)-P_g(x,t) - u_n(0,0)x_n$ is a viscosity solution of  
	\begin{equation*} 
		\left\{\begin{aligned}
			\overline{u}_t -x_n^{\gamma} \overline{F}(D^2 \overline{u},x,t ) & =\overline{f} && \text{in } Q_1^+  \\
			\overline{u} &= \overline{g} && \text{on } \partial_p Q_1^+ \cap \{x_n=0\},
		\end{aligned}\right.
	\end{equation*}
	where $\overline{F}(M,x,t)\coloneqq F(M + D^2 P_g,x,t) + x_n^{-\gamma}(f(0,0)-\partial_t P_g)$, $\overline{f}(x,t) \coloneqq f(x,t) - f(0,0)$, and $\overline{g}(x,t) \coloneqq g(x,t) - P_g(x,t) - u_n(0,0)x_n$ which satisfies
	\begin{align*}
		|\overline{f}(x,t)| &\leq \kappa (|x|+|t|^{\frac{1}{2-\gamma}})^{\alpha+\gamma} \quad \text{for all }(x,t) \in Q_1^+,\\
		|\overline{g}(x,t)| &\leq [g]_{C^{2,\alpha}(0,0)} (|x|+|t|^{\frac{1}{2-\gamma}})^{2+\alpha} \quad \text{for all }(x,t) \in \partial_p Q_1^+ \cap \{x_n=0\} .
	\end{align*}
	
	Define 
\begin{equation*}
	\overline{\beta}^1(x,t) \coloneqq \beta^1(x,t)  \quad \text{and} \quad \overline{\beta}^2(x,t) \coloneqq  \beta^1(x,t) \|D^2 P_g\|+ \beta^2(x,t) + x_n^{-\gamma}(f(0,0)-\partial_t P_g).
\end{equation*}
Then it follows that $|\overline{F}(M,x,t)- \overline{F}(M,0,0)| \leq \overline{\beta}^1(x,t) \|M\| + \overline{\beta}^2(x,t)$.
Let
	\begin{align*}
		K&\coloneqq \|\overline{u}\|_{L^{\infty}(Q_1^+)} + \delta^{-1}2^{\alpha^*+1}  2[\overline{\beta^2}]_{C^{\alpha^*}(0,0)} +\delta^{-1}2^{2+\alpha} \|\overline{g}\|_{C^{2,\alpha}(0,0)}  + \overline{F}(O_n,0,0) + \delta^{-1} 2^{\alpha+\gamma}\kappa,
	\end{align*}
by considering 
	\begin{align*}
		\widetilde{u}(x,t) \coloneqq \overline{u}(x,t)/K  \quad \text{and} \quad \widetilde{g}(x,t) \coloneqq \overline{g}(x,t)/K,
	\end{align*}
	and the scaling argument, we may assume that $\|u\|_{L^{\infty}(Q_1^+)} + |F(O_n,0,0)|\leq 1$, $u(0,0)=0$, $Du(0,0)=0$,
	\begin{align*} 
		|\beta^1(x,t)| &\leq \frac{\delta}{K2^{\alpha+1}} (|x|+|t|^{\frac{1}{2-\gamma}})^{\alpha} \quad \text{for all } (x,t) \in \overline{Q_1^+} ,   \\
		|\beta^2(x,t)| &\leq \frac{\delta}{2^{\alpha^*+1}} (|x|+|t|^{\frac{1}{2-\gamma}})^{\alpha^*} \quad \text{for all } (x,t) \in \overline{Q_1^+},  	\\
		|f(x,t)| &\leq \frac{\delta}{2^{\alpha+\gamma}} (|x|+|t|^{\frac{1}{2-\gamma}})^{\alpha+\gamma} \quad \text{for all } (x,t) \in \overline{Q_1^+},
	\end{align*}  
and 
	\begin{equation*} 
		|g(x,t)| \leq \frac{\delta}{2^{2+\alpha}} (|x|+|t|^{\frac{1}{2-\gamma}})^{2+\alpha} \quad \text{for all }  (x,t) \in  \partial_p Q_1^+ \cap \{x_n=0\}.
	\end{equation*}
	
Since $\alpha^* \leq \alpha$, if we apply \Cref{lem2:gen_c2a_sg} for $u$ and $\alpha^*$, we obtain a polynomial $\overline{P}(x)=\sum_{i=1}^n a_i x_i x_n$ such that $\overline{F}(D^2 \overline{P},0,0)=0$. Then $P(x,t) \coloneqq \overline{P}(x) + P_g(x,t) +u_n(0,0) x_n$ satisfies 
\begin{equation*}
	F(D^2 P,0,0) = 0, \quad \partial_t P =\partial_t P_g, \quad \text{and} \quad P(x',0,t) \equiv P_g(x,0,t).
\end{equation*}
The remaining proof is exactly the same as \Cref{thm:main2}.
\end{proof}

Finally, we prove the solvability of the Cauchy-Dirichlet problem, \Cref{thm:solvability} and \Cref{thm:solvability_s}.

\begin{proof}[Proof of \Cref{thm:solvability}]
For each $(x_0,t_0) \in \overline{\Omega^+_T}$, we just need to find a polynomial $P(x,t)$ with $\deg p \leq 2$ that satisfies $\|P\|_{C^2(\overline{\Omega_T^+}) } \leq C$ and
\begin{equation*}
	|u(x,t)-P(x,t)| \leq C(|x-x_0| + |t-t_0|^{1/2})^{2+\alpha} \quad \mbox{for all } (x,t) \in \Omega^+_T.
\end{equation*}
If $(x_0,t_0) \in \partial_p \Omega_T^+$, then our goal $P$ is just itself from \Cref{thm:main2} or \cite[Theorem 1.18]{LZ22}. It is enough to consider only the interior point $(x_0,t_0) \in \Omega_T^+$. Let $(\tilde{x}_0,\tilde{t}_0) \in \partial_p \Omega_T^+$ such that 
\begin{equation*}
	r\coloneqq \textnormal{dist}(\partial_p \Omega_T^+,(x_0,t_0)) = |x_0-\tilde{x}_0| +  |t_0-\tilde{t}_0|^{1/2}.
\end{equation*}
By the boundary $C^{2,\alpha_1}$-estimate (\Cref{thm:main2},  \cite[Theorem 1.18]{LZ22}), there exist a polynomial $P_1$ such that $\deg P_1 \leq 2$ and 
\begin{equation*}
	\left\{\begin{aligned}
		\partial_t P_1- \tilde{x}_{0n}^{\gamma} F(D^2P_1,\tilde{x}_0,\tilde{t}_0)&=0 && \text{if } \tilde{x}_{0n} >0 \\
		F(D^2P_1,\tilde{x}_0,\tilde{t}_0)&=0 && \text{if } \tilde{x}_{0n} =0,
	\end{aligned}\right.
\end{equation*}
and a constant $A>0$ such that 
\begin{equation*}
		|DP_1(x_0,t_0)| + \|D^2P_1\| +|\partial_t P_1| \leq A
	\end{equation*}
and
\begin{equation} \label{u-p1}
	|u(x,t)-P_1(x,t)|\leq A(|x-\tilde{x}_0|+ |t-\tilde{t}_0|^{1/2} )^{2+\alpha_1}
\end{equation}
for all $(x,t) \in \Omega_T^+ \cap Q_1^+(\tilde{x}_0,\tilde{t}_0)$ which is extensible throughout $\Omega_T^+$. Note that the function $v=u-P_1$ is a viscosity solution of 
\begin{equation*}
	v_t - x_n^{\gamma} \overline{F}(D^2 v,x,t) =0 \quad \text{in } \Sigma,
\end{equation*}
where $\overline{F} (M,x,t) = F(M +D^2P_1,x,t)$ and $\Sigma=\{(x,t):|x-x_0|+ |t-t_0|^{1/2} < r/2 \}$. Since there exists $\theta \in \mathbb{R}$ such that $\overline{F}(\theta I_n, x_0,t_0) =0$ and the function $w(x,t) =v(x,t) -\theta |x|^2/2$ satisfy $w_t -x_n^{\gamma} \overline{F}(D^2 w+ \theta I_n,x,t) = 0$, we may assume that $\overline{F}(O_n,x_0,t_0)=0$.

Since $C^{\alpha}(\overline{\Sigma}) \subset H^{\alpha}(\overline{\Sigma})$ and by applying the interior $C^{2,\alpha}$-estimate for uniformly parabolic equations to $v$ (see \cite[Theorem 1.1]{Wan92b}, \cite[Theorem 8.1]{CC95}), we have there exists a polynomial $P_2$ with $\deg P_2 \leq 2$ such that 
\begin{equation} \label{norm_p2}
	r |DP_2(x_0,t_0)| + r^2\|D^2P_2\| +r^2 |\partial_t P_2 | \leq B \left( \|v\|_{L^{\infty}(\Sigma)} + r^2 \| \overline{\beta}^2\|_{H^{\alpha}(\overline{\Sigma})} \right)
\end{equation}
and
\begin{equation} \label{v_int} 
	|v(x,t)-P_2(x,t)|  \leq B \bigg( \frac{ \|v\|_{L^{\infty}(\Sigma)}}{r^{2+\alpha}} + \frac{ \| \overline{\beta}^2\|_{H^{\alpha}(\overline{\Sigma})} }{r^{\alpha}} \bigg) (|x-x_0| + |t-t_0|^{1/2})^{2+\alpha} 
\end{equation}
for all $(x,t) \in \Sigma$, where $B>0$ is a constant depending only on $n$, $\lambda$, $\Lambda$, $\gamma$, $\alpha$, and $\|\beta^1\|_{H^{\alpha}(\overline{\Omega_T^+})}$ and $\overline{\beta}^2 \coloneqq  \beta^1 \|D^2P_1\| + \beta^2$.
From \eqref{u-p1}, we see
\begin{align} 
	|v(x,t)| &\leq A (|x-\tilde{x}_0|+|t-\tilde{t}_0|^{1/2} )^{2+\alpha_1} \leq A \big(|x-x_0|+|t-t_0|^{1/2} +r \big)^{2+\alpha} \nonumber \\
	& \leq AC r^{2+\alpha} \quad \text{for all }(x,t) \in \Sigma  \label{est_v}
\end{align}
and we also see
\begin{equation} \label{est_vt-Lv}
	|\overline{\beta}^2(x,t)| \leq C r^{\alpha} \quad \text{for all } (x,t) \in \overline{\Sigma}.
\end{equation}
Thus, combining \eqref{norm_p2}, \eqref{est_v}, and \eqref{est_vt-Lv} gives  
\begin{equation} \label{norm_pk2'}  	
	|P_2(x_0,t_0)|+ r |DP_2(x_0,t_0)| + r^2\|D^2P_2\| +r^2 |\partial_t P_2 | \leq ABC r^{2+\alpha}.
\end{equation}
Furthermore, combining \eqref{v_int}, \eqref{est_v}, and \eqref{est_vt-Lv} gives  
\begin{align*}
	|u(x,t)-P_1(x,t)-P_2(x,t)| &\leq ABC (|x-x_0| + |t-t_0|^{1/2})^{2+\alpha} \quad \text{for all } (x,t) \in \Sigma. 
\end{align*}

On the other hand, combining \eqref{u-p1} and \eqref{norm_pk2'} gives
\begin{align*}
	&|u(x,t)-P_1(x,t)-P_2(x,t)| \\
	&\quad\leq |u(x,t)-P_1(x,t)| + |P_2(x,t)| \\
	&\quad\leq A(|x-\tilde{x}_0|+ |t-\tilde{t}_0|^{1/2} )^{2+\alpha} +|P_2(x_0,t_0)|  +|DP_2 (x_0,t_0)| |x-x_0|  \\
	&\qquad\qquad\qquad\qquad\qquad\qquad\qquad\qquad +\frac{1}{2} \|D^2P_2\| |x-x_0|^2 + |\partial_t P_2| |t-t_0|\\
	&\quad\leq A(|x-x_0|+ |t-t_0|^{1/2} )^{2+\alpha} + ABCr^{2+\alpha}+ ABCr^{1+\alpha} |x-x_0|\\
	&\qquad\qquad\qquad\qquad\qquad\qquad\qquad\qquad+ \frac{1}{2}ABCr^{\alpha}|x-x_0|^2 + ABCr^{\alpha} |t-t_0|\\
	&\quad\leq  ABC(|x-x_0|+ |t-t_0|^{1/2} )^{2+\alpha} 
\end{align*}
for all $(x,t) \in \Omega_T^+ $ with $|x-x_0|+ |t-t_0|^{1/2} \ge r/2$. That is, $P \coloneqq P_1+P_2$ is the desired polynomial.
\end{proof}

\begin{proof}[Proof of \Cref{thm:solvability_s}]
For each $(x_0,t_0) \in \overline{\Omega^+_T}$, we just need to find a polynomial $P(x,t)$ with $\deg p \leq 2$ that satisfies $\|P\|_{C^2(\overline{\Omega_T^+}) } \leq C$ and
\begin{equation*}
	|u(x,t)-P(x,t)| \leq C(|x-x_0| + |t-t_0|^{\frac{1}{2-\gamma}})^{2+\alpha} \quad \mbox{for all } (x,t) \in \Omega^+_T.
\end{equation*}
If $(x_0,t_0) \in \partial_p \Omega_T^+$, then our goal $P$ is just itself from \Cref{thm:main2_sg} or \cite[Theorem 1.18]{LZ22}. It is enough to consider only the interior point $(x_0,t_0) \in \Omega_T^+$. Let $(\tilde{x}_0,\tilde{t}_0) \in \partial_p \Omega_T^+$ such that 
\begin{equation*}
	r\coloneqq \textnormal{dist}(\partial_p \Omega_T^+,(x_0,t_0)) = |x_0-\tilde{x}_0| +  |t_0-\tilde{t}_0|^{\frac{1}{2-\gamma}}.
\end{equation*}
Since $H^{\alpha}(\overline{\Omega^+_T}) \subset C^{\alpha}(\overline{\Omega^+_T})$ and by the boundary $C^{2,\alpha_1}$-estimate (\Cref{thm:main2_sg},  \cite[Theorem 1.18]{LZ22}), there exist a polynomial $P_1$ such that $\deg P_1 \leq 2$ and 
\begin{equation*}
	\left\{\begin{aligned}
		&\partial_t P_1- \tilde{x}_{0n}^{\gamma} F(D^2P_1,\tilde{x}_0,\tilde{t}_0) =f(\tilde{x}_0,\tilde{t}_0) && \text{if } \tilde{x}_{0n} >0 \\
		&F(D^2P_1,\tilde{x}_0,\tilde{t}_0) =0, \quad \partial_t P_1 = \partial_t P_g && \text{if } \tilde{x}_{0n} =0,
	\end{aligned}\right.
\end{equation*}
and a constant $A>0$ such that 
\begin{equation*}
		|DP_1(x_0,t_0)| + \|D^2P_1\| +|\partial_t P_1| \leq A
	\end{equation*}
and
\begin{equation} \label{u-p1_sg}
	|u(x,t)-P_1(x,t)|\leq A(|x-\tilde{x}_0|+ |t-\tilde{t}_0|^{\frac{1}{2-\gamma}} )^{2+\alpha^*}
\end{equation}
for all $(x,t) \in \Omega_T^+ \cap Q_1^+(\tilde{x}_0,\tilde{t}_0)$ which is extensible throughout $\Omega_T^+$. Note that the function $v=u-P_1$ is a viscosity solution of 
\begin{equation*}
	v_t - x_n^{\gamma} \overline{F}(D^2 v,x,t) = \overline{f} \quad \text{in } \Sigma,
\end{equation*}
where $\overline{F} (M,x,t) = F(M +D^2P_1,x,t) + x_n^{-\gamma}(f(x_0,t_0)-\partial_t P_1)$, $\overline{f}(x,t) = f(x,t) - f(x_0,t_0)$, and $\Sigma=\{(x,t):|x-x_0|+ |t-t_0|^{\frac{1}{2-\gamma}} < r/2 \}$.

 Since there exists $\theta \in \mathbb{R}$ such that $ \overline{F}(\theta I_n, x_0,t_0) = 0$ and the function $w(x,t) =v(x,t) -\theta |x|^2/2$ satisfy $w_t -x_n^{\gamma} \overline{F}(D^2 w+ \theta I_n,x,t) = \overline{f}$, we may assume that $\overline{F}(O_n,x_0,t_0)=0$.

By applying the interior $C^{2,\alpha}$-estimate for uniformly parabolic equations to $v$ (see \cite[Theorem 1.1]{Wan92b}, \cite[Theorem 8.1]{CC95}), we have there exists a polynomial $P_2$ with $\deg P_2 \leq 2$ such that 
\begin{align} \label{norm_p2_sg}
	& r |DP_2(x_0,t_0)| + r^2\|D^2P_2\| +r^2 |\partial_t P_2 |  \leq B \left( \|v\|_{L^{\infty}(\Sigma)} + r^2 \| \overline{f}\|_{H^{\alpha}(\overline{\Sigma})}  + r^2 \| \overline{\beta}^2\|_{H^{\alpha}(\overline{\Sigma})} \right)
\end{align}
and
\begin{align} \label{v_int_sg} 
	|v(x,t)-P_2(x,t)| &\leq B \bigg( \frac{ \|v\|_{L^{\infty}(\Sigma)}}{r^{2+\alpha}} + \frac{ \| \overline{f}\|_{H^{\alpha}(\overline{\Sigma})} + \| \overline{\beta}^2\|_{H^{\alpha}(\overline{\Sigma})}  }{r^{\alpha}}  \bigg) (|x-x_0| + |t-t_0|^{\frac{1}{2-\gamma}})^{2+\alpha} 
\end{align}
for all $(x,t) \in \Sigma$, where $B>0$ is a constant depending only on $n$, $\lambda$, $\Lambda$, $\gamma$, $\alpha$, and $\|\beta^1\|_{C^{\alpha}(\overline{\Omega_T^+})}$ and $\overline{\beta}^2 \coloneqq  \beta^1 \|D^2P_1\| + \beta^2$.
From \eqref{u-p1_sg}, we see
\begin{align} 
	|v(x,t)| &\leq A (|x-\tilde{x}_0|+|t-\tilde{t}_0|^{\frac{1}{2-\gamma}} )^{2+\alpha^*} \leq A \big(|x-x_0|+|t-t_0|^{\frac{1}{2-\gamma}} +r \big)^{2+\alpha} \nonumber \\
	& \leq AC r^{2+\alpha} \quad \text{for all }(x,t) \in \Sigma \label{est_v_sg}
\end{align}
and we also see
\begin{equation} \label{est_vt-Lv_sg}
\left\{\begin{aligned}
	|\overline{f}(x,t)| &\leq C r^{\alpha} \quad \text{for all } (x,t) \in \overline{\Sigma}\\
	|\overline{\beta}^2(x,t)| &\leq C r^{\alpha} \quad \text{for all } (x,t) \in \overline{\Sigma}.
\end{aligned} \right.
\end{equation}
Thus, combining \eqref{norm_p2_sg}, \eqref{est_v_sg}, and \eqref{est_vt-Lv_sg} gives  
\begin{equation} \label{norm_pk2'_sg}  	
	|P_2(x_0,t_0)|+ r |DP_2(x_0,t_0)| + r^2\|D^2P_2\| +r^2 |\partial_t P_2 | \leq ABC r^{2+\alpha}.
\end{equation}
Furthermore, combining \eqref{v_int_sg}, \eqref{est_v_sg}, and \eqref{est_vt-Lv_sg} gives  
\begin{align*}
	|u(x,t)-P_1(x,t)-P_2(x,t)| &\leq ABC (|x-x_0| + |t-t_0|^{\frac{1}{2-\gamma}})^{2+\alpha} \quad \text{for all } (x,t) \in\Sigma. 
\end{align*}

On the other hand, combining \eqref{u-p1_sg} and \eqref{norm_pk2'_sg} gives
\begin{align*}
	&|u(x,t)-P_1(x,t)-P_2(x,t)| \\
	&\quad\leq |u(x,t)-P_1(x,t)| + |P_2(x,t)| \\
	&\quad\leq A(|x-\tilde{x}_0|+ |t-\tilde{t}_0|^{\frac{1}{2-\gamma}} )^{2+\alpha^*} +|P_2(x_0,t_0)|  +|DP_2 (x_0,t_0)| |x-x_0|  \\
	&\qquad\qquad\qquad\qquad\qquad\qquad\qquad\qquad +\frac{1}{2} \|D^2P_2\| |x-x_0|^2 + |\partial_t P_2| |t-t_0|\\
	&\quad\leq A(|x-x_0|+ |t-t_0|^{\frac{1}{2-\gamma}} )^{2+\alpha} + ABCr^{2+\alpha}+ ABCr^{1+\alpha} |x-x_0|\\
	&\qquad\qquad\qquad\qquad\qquad\qquad\qquad\qquad+ \frac{1}{2}ABCr^{\alpha}|x-x_0|^2 + ABCr^{\alpha} |t-t_0|\\
	&\quad\leq  ABC(|x-x_0|+ |t-t_0|^{\frac{1}{2-\gamma}} )^{2+\alpha} 
\end{align*}
for all $(x,t) \in \Omega_T^+ $ with $|x-x_0|+ |t-t_0|^{\frac{1}{2-\gamma}} \ge r/2$. That is, $P \coloneqq P_1+P_2$ is the desired polynomial.
\end{proof}

%
%

\end{document}